\documentclass{amsart}
\usepackage{amsthm,amssymb,amsmath,hyperref, esint}
\usepackage{graphicx}
\usepackage{color}
\usepackage{tikz}
\usetikzlibrary{decorations.pathreplacing}
\usepackage[utf8]{inputenc}
\usepackage{textgreek}
\usepackage{mathtools}

\definecolor{darkorchid}{rgb}{0.6, 0.2, 0.8}

\theoremstyle{plain}
\newtheorem{thm}{Theorem}[section]

\newtheorem{lemma}[thm]{Lemma}
\newtheorem{prop}[thm]{Proposition}

\newcommand{\Span}{\mathrm{Span}}
\theoremstyle{definition}
\newtheorem{defn}[thm]{Definition}

\theoremstyle{remark}
\newtheorem{rmk}[thm]{Remark}

\theoremstyle{plain}

\numberwithin{equation}{section}

\newcommand{\NN}{\mathbb{N}}
\newcommand{\ZZ}{\mathbb{Z}}

\newcommand{\del}{\delta}

\renewcommand{\phi}{\varphi}

\def\udot#1{\ifmmode\oalign{$#1$\crcr\hidewidth.\hidewidth
    }\else\oalign{#1\crcr\hidewidth.\hidewidth}\fi}

\def\RR{\mathbb{R}}
\def\ZZ{\mathbb{Z}}
\def\TT{\mathbb{T}}

\def\QQ{\mathbb{Q}}

\usepackage{float}

\usepackage{geometry}
\title{Arbitrary finite intersections of doubling measures and applications}
\author{Theresa C. Anderson}
\address{(Anderson) Department of Mathematics, Carnegie Mellon University}
\author{Elisa Bellah}
\address{(Bellah) Department of Mathematics, Carnegie Mellon University}
\author{Zoe Markman}
\address{(Markman) Department of Mathematics, Swarthmore College}
\author{Teresa Pollard}
\address{(Pollard) Department of Mathematics, New York University}
\author{Josh Zeitlin}
\address{(Zeitlin) Department of Mathematics, Yale University}
\begin{document}
\maketitle
\author
\begin{abstract}
Using a wide array of machinery from diverse fields across mathematics, we provide a construction of a measure on the real line which is doubling on all $n$-adic intervals for any finite list of $n\in\NN$, yet not doubling overall. In particular, we extend previous results in the area, where only two coprime numbers $n$ were allowed, by using substantially new ideas. In addition, we provide several nontrivial applications to reverse H\"older weights, $A_p$ weights, Hardy spaces, BMO and VMO function classes, and connect our results with key principles and conjectures across number theory.
\end{abstract}


\section{Introduction}

A longstanding folkloric question in harmonic analysis is to construct a measure $\mu$ on the real line that is $n$-adic doubling for every $n$, yet not doubling overall (see Section \ref{two bases} for definitions).  Analogous, and similarly longstanding, analogues of this question exist for other ubiquitous objects in analysis, such as the reverse H\"older weight classes, Muckenhoupt $A_p$ weight classes, bounded mean oscillation (BMO) function classes, Hardy spaces ($H^1$) and the vanishing mean oscillation (VMO) function class. These classes have been extensively studied and we refer the reader to \cite{CN}, \cite{Cruz}, \cite{G}, \cite{GJ}, \cite{LPW}, and \cite{S}, among others, for the requisite background. 

Outside of unpublished lecture notes of Peter Jones from 1978 showing that $BMO_2\cap BMO_3 \neq BMO$, the authors are aware of little progress toward these questions until 2019. (The only reference to these notes that we are aware of in the literature come from a paper of Krantz \cite{Krantz}; a sketch of Jones's results is given, but details are sparse.) Then in 2019, Boylan, Mills and Ward explicitly constructed a measure on $[0,1]$ that was dyadic and triadic doubling, yet not doubling overall. This involved a careful re-weighting procedure and technical calculations; moreover the result relied on number theory that did not generalize beyond the primes $2$ and $3$ (corresponding to the \emph{bases} $2$ and $3$ for dyadic and triadic intervals). In 2022, Anderson and Hu were able to extend the Boylan-Mills-Ward construction to any distinct prime bases $p$ and $q$, which required several new ideas. Notably, the number theory behind the bases, though elementary, was significantly more intricate, along with the geometry behind selecting intervals on which to re-weight the measure. Additionally, a new procedure to avoid the specific calculations of \cite{BMW} was developed. Anderson, Travesset and Veltri extended this construction to coprime bases, which only involved changing the number theoretic part of the construction \cite{ATV}. However, there was still a significant gap in resolving the folkloric question about measures. Here, we tighten this gap as well as provide a wide-range of new applications for all the weight and function classes mentioned above. Namely, we are able to extend the results of \cite{AH} to any finite set of \emph{arbitrary} bases (which includes any two bases, greatly generalizing \cite{ATV}). Moreover, our proofs use a novel interplay of number theory, topology, group theory, measure theory, geometry, harmonic and functional analysis. Finally we make precise an analogy to the \emph{Hasse principle} from arithmetic geometry in this harmonic analytic setting.  

Our first contribution is to extend the Anderson-Hu construction to any two bases, which relies in part on lemmas from \cite{AH} that are generalized herein, rather than the more technical machinery from \cite{ATV}. However, the main contribution is not from the more general elementary number theory developed here, but on the interval selection used to re-weight the measure. A key part of \cite{AH} involved showing that specific distinguished points $Y(J)$ on $p$-adic intervals $J$ and $Z(I)$ on $q$-adic intervals $I$ were extremely close. That is, for any $\varepsilon>0$ we have 
\begin{equation}
\label{key closeness criterion}
    |Y(J)-Z(I)|<\varepsilon|I|.
\end{equation}
While this inequality can be shown to hold for coprime bases, it does not extend to the general two bases case; our insight is to replace $Y(J)$ with a revolving distinguished point $\zeta$ that may change from interval to interval, yet still satisfies the closeness relationship. Moreover, our new interval selection and re-weighting procedure automatically generalizes to a finite list of multiples of the larger base. This variant shows that there is a dyadic, triadic and $6$-adic doubling measure that is not doubling, as an example, complimenting results in \cite{AH2}, \cite{LPW}, \cite{PWX}, \cite{TM} \cite{Wu}. In attempting to extend this construction to three arbitrary bases we discovered an observation connected with \emph{far numbers} that illustrates the limit of this construction and justifies the need for a new one. A detailed number theoretic explanation that our construction would be difficult to generalize farther is given in the appendix.

This leads to our second main contribution -- constructing a measure that is $n$-adic doubling for any finite list of $n$, yet not overall.  This result generalizes all previous described work, while using a significantly different argument.
\begin{thm} 
    \label{finite bases theorem}
    There exists an infinite family of measures that are simultaneously $n_i$-adic doubling for any finite list of positive integers $n_1, \dots, n_k$, but not doubling overall.
\end{thm}

The new process involves firstly using the whole real line, instead of just $[0,1]$ to select distinguished dyadic intervals on which to reweight our measure (though later on we are able to compactify our results). This process stems from relaxing the restriction \eqref{key closeness criterion} by adding absolute values. Namely, let $\alpha\in\NN$ and $\varepsilon< 2^{-2\alpha}$. We show that inside $[\alpha,\alpha+1]$ there exists some arbitrarily large $x\in\NN$ such that
\[\left|\frac{1}{2^x}-\frac{1}{n_i^{[x\log_{n_i}2]}}\right|<\frac{\varepsilon}{2^x}\] for all $n_i$ in our finite list, where the brackets denotes the closest integer function. In this new construction we are choosing the sizes of each interval simultaneously, as opposed to \cite{AH} where they are chosen independently. The proof of the existence of such an $x$ involves an intricate interplay of topology, number theory and geometry, and has deep connections with Schanuel's conjecture (an outstanding open problem in transcendental number theory \cite{Sch}).  Finally, we make partial progress toward the infinite base case and show the significant obstacles that would need to be overcome for this case to be true.

Both of these constructions lead to our third contribution of applications to weight and function classes.  Notably, we are able to weave previous progress initiated in \cite{AH} with both functional analysis and our new constructions to show that any finite intersection of corresponding weight and function classes is never the whole class. We show

\begin{thm}
\label{application theorem}
   Let $S \subset \NN$ be finite. Then we have:
    \begin{enumerate}
        \item $\bigcap_{n \in S}RH_r^n \neq RH_r$
        \item $\bigcap_{n\in S} A_p^n\neq A_p$
        \item $\bigcap_{n \in S}BMO_n \neq BMO$
        \item $\sum_{n\in S}H^1_n \neq H^1$
        \item $\bigcap_{n \in S}VMO_n \neq VMO$
    \end{enumerate}
    Precise definitions of all classes are given in Section \ref{applications}.  All these results hold over both $\RR$ and $\TT = [0,1]$.  In particular, these results hold for sums and intersections involving two arbitrary natural numbers.
\end{thm}

These results are different from, and complimentary to, the array of results in \cite{LPW}, where the authors show that the intersection (or sum) of two generalized dyadic function classes is indeed the full function class (see also \cite{AHJOW} for discussion and extension of their results). Thus, our results illustrate the subtle difference among generalized $n$-adic grids for a single $n$ and intersections of standard $n$-adic grids. Finally, the proofs revolve around a construction stemming from our measure, which is modified and interwoven with key relationships between the weight and function classes, both specifically and in a general functional analysis sense.

We write toward a broad audience and do not assume any deep familiarity with the techniques in \cite{AH}, but rather introduce tools and concepts as we go. Several times we simultaneously give details of intermediate steps in \cite{AH} while generalizing them, while other times we opt for a non-technical description of steps which are essentially identical here. Armed with this big picture, the interested reader can then refer to \cite{AH} for more details. We include several remarks to put the wide array of results and techniques in context, as well as to provide intuition.

This paper is organized as follows. In Section \ref{two bases} we simultaneously generalize and streamline the argument in \cite{AH} as to apply to a wider range of cases, including the general two bases case. The next section, Section \ref{finite bases}, begins with a connection to \emph{far numbers}, indicating that our general construction in the previous section does not generalize.  From there we outline our new construction to get Theorem \ref{finite bases theorem}. Finally, we outline partial progress toward the infinite base case and the obstacles to overcome; as well as briefly discuss the relationship with this line of work and the \emph{Hasse principle}. The final section, Section \ref{applications}, contains all the applications to weight and function classes, proving Theorem \ref{application theorem}.

\subsection{Acknowledgements}
This research, and authors Anderson, Markman, Pollard and Zeitlin, were partially supported by NSF DMS-2231990 and NSF CAREER  DMS-2237937 (T.C. Anderson). In addition, Markman, Pollard and Zeitlin benefited from funding from Carnegie Mellon's SUAMI program.  The authors also thank David Cruz-Uribe for providing several references.

\section{Extension of the Anderson-Hu Construction}
\label{two bases}

 Our main goal in this section will be to extend the Anderson-Hu construction to the general two bases case. We first recall several definitions and set some notation.

\begin{defn} \label{doubling def}
    A \emph{doubling measure} $\mu$ is a measure for which there exists a positive absolute constant $C$ such that for every interval $I \subset \RR$ we have 
    \[\mu(2I) \le C\mu(I),\] where $2I$ is the interval which shares the same midpoint of $I$ and twice the length of $I$. Equivalently, $\mu$ is doubling if and only if there is a positive absolute constant $C$ so that \[\frac{1}{C}\leq \frac{\mu(I_1)}{\mu(I_2)}\leq C\] 
    for any two adjacent intervals $I_1$ and $I_2$ of equal length. 
\end{defn}

\begin{defn} \label{adic def}
For a positive integer $n$, the \emph{standard $n$-adic system} $\mathcal{D}(n)$ is the collection of $n$-adic intervals in $\mathbb{R}$ of the form
\begin{equation} \label{eq1} 
I=\left[ \frac{k-1}{n^m}, \frac{k}{n^m} \right), \quad m, k \in \mathbb{Z}. 
\end{equation}
For an $n$-adic interval $I$ defined as in (\ref{eq1}), the \emph{$n$-adic children} of $I$ are given by
\[ 
I_j=\left[ \frac{k-1}{n^m}+\frac{j-1}{n^{m+1}}, \frac{k-1}{n^m}+\frac{j}{n^{m+1}} \right), \quad 1 \le j \le n. 
\]
We say that $n$-adic intervals $J_1$ and $J_2$ are \emph{$n$-adic siblings} if they are the children of some common $n$-adic interval $I$. Note that an arbitrary pair of $n$-adic intervals are not necessarily $n$-adic siblings.
\end{defn}

\begin{defn} \label{adic doubling def}
    A measure $\mu$ is an \emph{n-adic doubling measure} if there exists a positive absolute constant $C$ so that \[\frac{1}{C}\leq \frac{\mu(J_1)}{\mu(J_2)}\leq C\]
    for any two $n$-adic siblings $J_1$ and $J_2$. 
\end{defn}
It should be noted that any doubling measure is automatically $n$-adic doubling for every $n \in \NN$. Furthermore, it can be quickly checked that a measure is $n$-adically doubling if and only if it is $n^k$-adically doubling for every positive integer $k$. Further details on doubling and related concepts can be found in \cite{AH}. \\

Finally, we set some notation. Given an $n$-adic interval $I$ as in (\ref{eq1}), we write $Y(I)$ be the right endpoint of $I_1$, that is
\begin{equation} \label{yidef}
Y(I)=\frac{k-1}{n^m}+\frac{1}{n^{m+1}},
\end{equation}
and $Z(I)$ be the left endpoint of $I_n$, that is 
\begin{equation} \label{zidef}
Z(I)=\frac{k-1}{n^m}+\frac{n-1}{n^{m+1}}. 
\end{equation}
as illustrated in Figure \ref{notation}.

\begin{figure}[ht]
\begin{tikzpicture}[scale=8.5]
\label{notation}
\draw (0,0) -- (1,0); 
\draw (0, -0.02) -- (0, 0.02);
\draw (0.15, -0.02) -- (0.15, 0.02);
\draw (0.3, -0.02) -- (0.3, 0.02);
\draw (0.7, -0.02) -- (0.7, 0.02);
\draw (0.85, -0.02) -- (0.85, 0.02);
\draw (1, -0.02) -- (1, 0.02);
\fill (0, -0.015) node [below] {$\frac{k-1}{n^m}$};
\fill (1, -0.015) node [below] {$\frac{k}{n^m}$};
\fill (0.075, 0.005) node [above] {$I_1$}; 
\fill (0.225, 0.005) node [above] {$I_2$}; 
\fill (0.5, 0.005) node [above] {$\cdots$}; 
\fill (0.925, 0.005) node [above] {$I_n$}; 
\fill (0.15,0) circle [radius=.2pt];
\fill (0.85,0) circle [radius=.2pt];
\fill (0.15, -0.02) node [below] {$Y(I)$};
\fill (0.85, -0.015) node [below] {$Z(I)$};
\end{tikzpicture}
\caption{Illustration of distinguished points $Y(I)$ and $Z(I)$}
\end{figure}
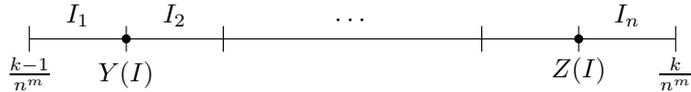

Our main result of this section extends the Anderson-Hu construction to show the following.

\begin{thm}
    \label{two bases theorem}
    There exists an infinite family of measures that are both $u$-adic and $v$-adic doubling for any pair of positive integers $u$ and $v$, but not doubling overall.
\end{thm}

 Our key insight in proving Theorem \ref{two bases} is a new interval selection procedure to that in \cite{AH}, where the intervals that are selected will be given a modified measure, with all other intervals given Lebesgue measure.  The modified measure is an iterative reweighting, which is unchanged from the construction of Anderson-Hu (having its genesis in \cite{BMW}).  Therefore, before describing our new contributions, we begin with a description of the re-weighting procedure from \cite{AH}, in nontechnical terms, so that the reader can build intuition without having to read \cite{AH} and also convince themselves that this procedure will give the desired results on our selected intervals.
 
\subsection{Description of the Anderson-Hu construction}
\label{measure construction}
In \cite{AH}, the authors construct a measure which is $p$-adic and $q$-adic doubling for distinct primes $p$ and $q$ but not doubling overall. To do this, they select an infinite family of $q$-adic intervals $I^{\alpha}$ indexed by parameters $\alpha\in\NN$. Then they re-weight the Lebesgue measure on the $q$-adic children of $I^\alpha$ to the right and left of the distinguished point $Z(I^\alpha)$, as defined in (\ref{zidef}). \\

We give a description of the reweighting procedure here. Choose any $a<1$ and $b>1$ satisfying $a(q-1)+b=q$. We reweight each interval $I^{\alpha}$ iteratively in $2 \alpha$ steps as follows. For convenience, set $I=I^{\alpha}$. First, redistribute the Lebesgue measure on the $q$ children of $I$ by setting
\[\mu(I_q)=b |I_q|=\frac{b|I|}{q} \text{ and } \mu(I_j)=a |I_j|=\frac{a|I|}{q}\]
for $j=1, \dots, q-1$. Next, set $H^{(1)}: =I_{q-1}$ and $G^{(1)}:=I_q$ and redistribute the measure to the children of $H^{(1)}$ and $G^{(1)}$ as in the previous step. That is, set
\[\mu(H^{(1)}_q)=\frac{b \mu(H^{(1)})}{q} \text{ and }
\mu(H^{(1)}_j)=\frac{a \mu(H^{(1)})}{q}\]
and similarly 
\[\mu(G^{(1)}_q)=\frac{b \mu(G^{(1)})}{q}\text{ and }
\mu(G^{(1)}_j)=\frac{a \mu(G^{(1)})}{q}\]
for $j=1, \dots, q-1$. Next, set $H^{(2)}:=H^{(1)}_q$ and $G^{(2)}:=G^{(1)}_1$ and repeat the reweighting procedure on $H^{(2)}$ and $G^{(2)}$ from the previous step. We then iterate in this way for $\alpha$ steps. That is, we set
\[\mu(H^{(k)}_q)=\frac{b \mu(H^{(k)})}{q} \text{ and }
\mu(H^{(k)}_j)=\frac{a \mu(H^{(k)})}{q}\]
\[\mu(G^{(k)}_q)=\frac{b \mu(G^{(k)})}{q}\text{ and }
\mu(G^{(1)}_j)=\frac{a \mu(G^{(k)})}{q}\]
for $j=1, \dots, q-1$ and then define $H^{(k+1)}:=H^{(k)}_q$ and $G^{(k+1)}:=G^{(k)}_1$ 
for $k=1, \dots, \alpha$
as illustrated in the figure below

\begin{figure}[ht]
\begin{tikzpicture}[scale=5.1]
\draw (-.8,0) -- (1.5,0); 
\fill (1.5,0) circle [radius=.2pt];
\fill (1.5, -0.01) node [right] {{\footnotesize $\textZeta$}};
\fill (1.3,0) circle [radius=.2pt];
\fill (1,0) circle [radius=.2pt];
\fill (.4,0) circle [radius=.2pt];
\fill (-.6,0) circle [radius=.2pt];
\draw [decorate,decoration={brace,amplitude=8pt,raise=10pt},yshift=1.5pt] (1.3,0) -- (1.5,0) node [black,midway,xshift=0cm, yshift=.9cm] { {\tiny $H^{(2\alpha)}$}};
\draw [decorate,decoration={brace,amplitude=8pt,raise=10pt},yshift=1.5pt] (1.0,0) -- (1.3,0) node [black,midway,xshift=0cm, yshift=.9cm] { {\tiny $E^{(2\alpha-1)}$}};
\draw [decorate,decoration={brace,amplitude=8pt,raise=10pt},yshift=1.5pt] (.4,0) -- (1.0,0) node [black,midway,xshift=0cm, yshift=.9cm] { {\tiny $E^{(2\alpha-2)}$}};
\draw [decorate,decoration={brace,amplitude=8pt,raise=10pt},yshift=1.5pt] (-.6,0) -- (0.4,0) node [black,midway,xshift=0cm, yshift=.9cm] { {\tiny $E^{(2\alpha-3)}$}};
\fill (1.4, -.01) node [below] {{\color{blue}{\tiny $1$}}}; 
\fill (1.15, -.01) node [below] {{\color{blue}{\tiny $q-1$}}};
\fill (.7, -.01) node [below] {{{\color{blue}\tiny $q(q-1)$}}};
\fill (-.1, -.01) node [below] {{{\color{blue} \tiny $q^2(q-1)$}}};
\fill (1.4, .01) node [above] {{{\color{red} \tiny $a^\alpha b^\alpha$}}}; 
\fill (1.15, .01) node [above] {{{\color{red} \tiny $a^{\alpha+1} b^{\alpha-1}$}}};
\fill (.7, .01) node [above] {{{\color{red} \tiny $a^{\alpha+1} b^{\alpha-2}$}}};
\fill (-.1, .01) node [above] {{{\color{red}\tiny $a^{\alpha+1} b^{\alpha-3}$}}};
\draw (-.8,-.6) -- (1.5, -.6); 
\fill (1.5,-.6) circle [radius=.2pt];
\fill (1.5, -.61) node [right] {{\footnotesize $\textZeta$}};
\fill[black] (1.2,-.6) circle [radius=.4pt];
\fill (1.45, -.6) circle [radius=.2pt];
\draw [decorate,decoration={brace,amplitude=8pt,raise=10pt},yshift=4pt] (1.2,-.6) -- (1.5,-.6) node [black,midway,xshift=0cm, yshift=.9cm] { {\tiny $H^{(\alpha)}$}};
\draw [decorate,decoration={brace,amplitude=5pt,raise=10pt},yshift=1.5pt] (1.2,-.6) -- (1.45,-.6) node [black,midway,xshift=0cm, yshift=.7cm] { {\tiny $E^{(\alpha)}$}};
\fill (.4, -.6) circle [radius=.2pt]; 
\draw [decorate,decoration={brace,amplitude=8pt,raise=10pt},yshift=4pt] (.4,-.6) -- (1.2,-.6) node [black,midway,xshift=0cm, yshift=.9cm] { {\tiny $E^{(\alpha-1)}$}};
\fill (1.33, -.59) node [above] {{\color{red}\tiny $a^{\alpha+1}$}};
\fill (0.65, -.6) circle [radius=.2pt]; 
\fill (.55, -.61) node [below] {{\color{blue}\tiny $q^\alpha$}};
\fill (.925, -.61) node [below] {{\color{blue} \tiny $q^\alpha(q-2)$}}; 
\draw [decorate,decoration={brace,amplitude=8pt,raise=10pt},yshift=-3pt] (1.2,-.6) -- (.4,-.6) node [black,midway,xshift=0cm, yshift=-.9cm] {{{\color{blue}\tiny $q^\alpha(q-1)$}}};
\draw [decorate,decoration={brace,amplitude=8pt,raise=10pt},yshift=-3pt] (1.45,-.6) -- (1.2,-.6) node [black,midway,xshift=0cm, yshift=-.9cm] {{{\color{blue}\tiny $q^{\alpha-1}(q-1)$}}};
\fill (.55, -.59) node [above] {{\color{red}\tiny $ba^{\alpha-1}$}}; 
\fill (.925, -.59) node [above] {{\color{red} \tiny $a^\alpha$}}; 
\fill (-.7, -.6) circle [radius=.2pt]; 
\fill (-.4, -.6) circle [radius=.2pt]; 
\fill (-.55, -.61) node [below] {{\color{blue}\tiny $q^{\alpha+1}$}};
\fill (0, -.61) node [below] {{\color{blue}\tiny $q^{\alpha+1}(q-2)$}};
\draw [decorate,decoration={brace,amplitude=8pt,raise=10pt},yshift=-3pt] (.4,-.6) -- (-.7,-.6) node [black,midway,xshift=0cm, yshift=-.9cm] { {{\color{blue}\tiny $q^{\alpha+1}(q-1)$}}};
\draw [decorate,decoration={brace,amplitude=8pt,raise=10pt},yshift=4pt] (-.7,-.6) -- (.4,-.6) node [black,midway,xshift=0cm, yshift=.9cm] { {\tiny $E^{(\alpha-2)}$}};
\fill (-.55, -.59) node [above] {{\color{red}\tiny $ba^{\alpha-2}$}};
\fill (0, -.59) node [above] {{\color{red}\tiny $a^{\alpha-1}$}};
\end{tikzpicture}
\caption{Illustrated is $\mu$ on the left hand side of $\textZeta$.  Weightings are listed in red and lengths in blue.}
\label{20200811Fig01}
\end{figure}
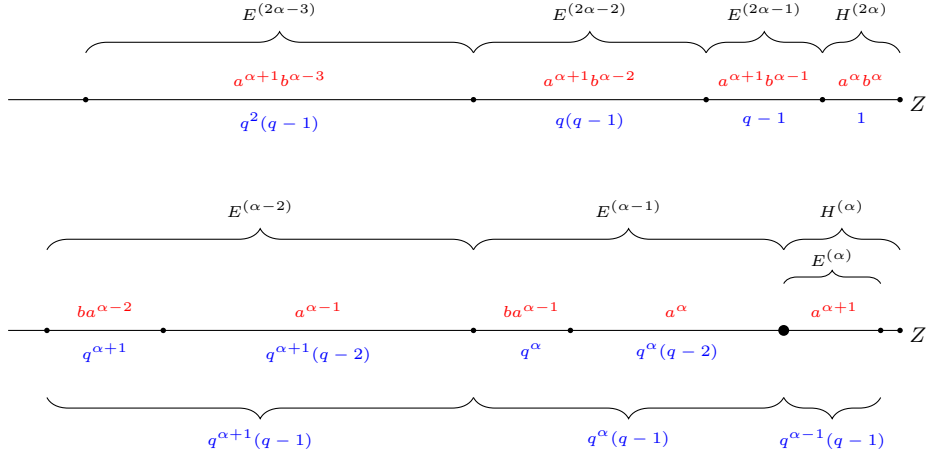

Observe that the first $\alpha$ steps of this procedure gives 
\[\mu(H^{(\alpha)})=\frac{a^\alpha |I|}{q^\alpha} \text { and }
\mu(G^{(\alpha)})=\frac{b^\alpha |I|}{q^\alpha},\]
as calculated in \cite{AH}. Note further that the intervals $H^{(\alpha)}$ and $G^{(\alpha)}$ are adjacent and of equal length. Hence, if we perform this process on infinitely many intervals $I^\alpha$ indexed by infinitely many positive integers $\alpha$, we see that there is no universal constant bounding the ratio of $\mu(G^{(\alpha)})$ and $\mu(H^{(\alpha)})$ for all $\alpha$, thus breaking the doubling property as defined in \ref{doubling def}. However, re-weighting the $q$-adic intervals in this way preserves $q$-adic doubling as the ratio of the measures of any $q$-adic siblings is either $1,$ $\frac{a}{b}$ or $\frac{b}{a}$. Note here that $G^{(\alpha)}$ and $H^{(\alpha)}$ are not $q$-adic siblings, but they are adjacent $q$-adic intervals of equal length.

After $\alpha$ steps, we now have a measure which is $q$-adic doubling, but not doubling overall. To allow for a measure which is simultaneously $p$-adic doubling, we reverse the above procedure on each interval $I^\alpha$ for an additional $\alpha$ steps. That is, for $k = \alpha+1, \dots, 2\alpha$ we set
\[\mu(H^{(k)}_q)=\frac{a \mu(H^{(k)})}{q}\text{ and }
\mu(H^{(k)}_j)=\frac{b \mu(H^{(k)})}{q}\] 
\[\mu(G^{(k)}_q)=\frac{b \mu(G^{(k-1)})}{q}\text{ and }
\mu(G^{(k)}_j)=\frac{a \mu(G^{(k)})}{q}.\]

Observe that these additional $\alpha$ steps do not have any effect on the $q$-adic doubling (as the ratio of $q$-adic intervals are still either $1$, $\frac{a}{b}$ or $\frac{b}{a}$) or the failure of doubling overall (since the measures of $H^{(\alpha)}$ and $G^{(\alpha)}$ are unchanged) but it does help to guarantee $p$-adic doubling. In \cite{AH}, the authors select intervals $I^\alpha$ so that the difference between $Y(J^\alpha)$ and $Z(I^\alpha)$ is suitably bounded, where $J^\alpha$ is the smallest $p$-adic interval containing $I^\alpha$. The reverse weighting procedure from steps $\alpha+1$ to $2\alpha$ guarantees $p$-adic doubling due to the \emph{exhaustion procedure} in \cite{AH}, where the closeness of $Y(J^\alpha)$ and $Z(I^\alpha)$ is key: essentially on these scales the $p$-adic interval must not span too many $q$-adic children with differing weights.

The existence of a family of intervals $I^\alpha$ on $[0, 1]$ so that the difference between $Y(J^\alpha)$ and $Z(I^\alpha)$ is suitably bounded, and thus the existence of the desired measure, is guaranteed by the following Proposition. 

\begin{prop}[Proposition 3.5 in \cite{AH}]
    Given any interval $\widetilde{J}\subseteq [0,1]$ and any $\varepsilon>0,$ there exists a $q$-adic interval $I\subset \widetilde{J}$ such that \[0<Y(J)-Z(I)\leq \varepsilon|I|\] where $J$ is the smallest $p$-adic interval containing $I.$
\end{prop}
 The existence of such a pair of intervals in Proposition 3.5 relies on the number theory explored in \cite{AH} which guarantees that given some sufficiently nice $p$-adic point of the form $\frac{k}{p^a}$ there exists an infinite sequence of points $\frac{j}{q^b}$ such that \[\frac{k}{p^a}-\frac{j}{q^b}=\frac{1}{p^aq^b}.\]
 If we let $J=\left[\frac{k-1}{p^a},\frac{k-1+p}{p^a}\right)$ be a $p$-adic interval then we can choose a large enough $b$ such that \[I=\left[\frac{j-q+1}{q^b},\frac{j+1}{q^b}\right)\subseteq J\] and $J$ is the largest $p$-adic interval containing $I.$  This framework was extended in \cite{ATV} to allow for such a sequence of $m$ and $n$-adic points with $\gcd(m,n)=1$.  Extending and simplifying this framework is thus our starting point.  
 
\subsection{Extending the construction}

Before proving our main result of this section, we first show how a simple generalization of \cite{AH} can extend the construction to the $q$ and $q^np^m$ bases case, where $p,q$ 
are distinct primes. Beginning here allows us to recall the techniques of \cite{AH} which we will alter to prove Theorem \ref{two bases theorem}. Due to the discussion above, the existence of a measure which is $q$-adic and $q^np^m$-adic doubling, for distinct primes $q$ and $p$, follows immediately from the construction of the measure outlined in Subsection \ref{measure construction} on the selected intervals $I^{(\alpha)}$ and the following proposition.
\begin{prop}
    Let $p$ and $q$ be primes with $p>q$. Given any interval $\Tilde{J} \subset [0,1]$ and any $\varepsilon > 0$, there exists a $q$-adic interval $I \subset \Tilde{J}$ such that $0<Y(J)-Z(I)\leq \varepsilon|I|$, where $J$ is the smallest $p^mq^n$-adic interval that contains $I$, for any fixed positive integers $m$ and $n$.
\end{prop}

\begin{proof}
    Fix integers $m, n$, interval $\tilde{J} \subset [0,1]$, and $\varepsilon >0$, and let $J'$ be the largest $q$-adic interval contained in $\tilde{J}$. Say that $J'$ has length $\dfrac{1}{q^{m'}}$. As in \cite{AH} let $m(p, q)$ denote the smallest positive integer satisfying
    \[q^{p-1} \not\equiv 1(\text{mod} \,p^{m(p, q)+1})\]
    and let $O_\ell(p, q)$ denote the order of $q^{p-1}$ in $(\ZZ/p^\ell\ZZ)^\times$ for any nonnegative integer $\ell$.
By Proposition 2.1 of \cite{AH}, when $\ell$ is sufficiently large there exists a constant $C(p, q)$ so that
\[\frac{O_\ell(p, q)}{p^{\ell-1}}=\frac{1}{p^{C(p, q)}}.\]
Choose a large enough integer $m_1$ so that all of the following hold
\[m_1 > \max\left\{\frac{m(p,q)}{q-1},m',\frac{m' + C(p,q) +1}{p-1}\right\},\]
$m \mid m_1$ and $\dfrac{1}{p^{m_1(q-1)}}< \varepsilon$. Next, choose a positive integer $k$ with
\[k \equiv 1 (\bmod (p^mq^n)^{C(p,q)+1})\]
so that for $m_1$ large enough we have $k/(p^mq^n)^s \in J'$ where $s = {m_1(q-1)}/{m}$, and let
\begin{align*}
    J = \left[\frac{k-1}{(p^mq^n)^{s}}, \frac{k-1+p^mq^{n}}{(p^mq^n)^{s}}\right].
\end{align*} Observe that $J$ is a $p^mq^n$-adic interval of length $1/(p^mq^n)^{s-1}$. Indeed, if we write
\[k=\kappa \cdot (p^m q^n)^{C(p, q)+1}+1\] 
for some $\kappa \in \ZZ$, then
\[\frac{k-1}{(p^mq^n)^{s}} = \frac{\kappa \cdot (p^mq^n)^{C(p,q)+1}+1-1}{(p^mq^n)^{s}} = \frac{\kappa\cdot(p^mq^n)^{C(p,q)}}{(p^mq^n)^{s-1}}\] and 

\[\frac{k-1+p^mq^{n}}{(p^mq^n)^{s}} = \frac{\kappa\cdot(p^mq^n)^{C(p,q)+1}+1-1+p^mq^n}{(p^mq^n)^{s}} = \frac{\kappa\cdot(p^mq^n)^{C(p,q)}+1}{(p^mq^n)^{s-1}}\].\\\

Also note that that \[Y(J) = \frac{k}{(p^mq^n)^{s}} \in J'\]
by choice of $k$. Since we chose $p>q$ and $m_1$ large enough, we see that $Y(J)$ is in the interior of $J'$. So, if we choose $m_1$ sufficiently large, we can make the interval $J$ arbitrarily small to ensure that $J \subseteq J'$. By Proposition 2.3 of \cite{AH}, there are infinitely many pairs $m_2, j \in \NN$ such that 
\[\frac{k}{p^{m_1(q-1)}}-\frac{j}{q^{m_2(p-1)}}=\frac{1}{p^{m_1(q-1)}q^{m_2(p-1)}}\]
with $j \equiv -1 (\bmod \, q)$. We now choose such an $m_2$ subject to the additional constraint that 
\[q^{m_2(p-1)}>10q\cdot p^{m_1(q-1)},\] and set \[I = \left[ \frac{j+1-q}{q^{n s +m_2(p-1)}}, \frac{j+1}{q^{n s +m_2(p-1)}}\right).\]  By our choice of $j \equiv -1 (\bmod \,q)$ we can see that $I$ is a $q$-adic interval. Furthermore, we have
\[Z(I)=\frac{j}{q^{ns +m_2(p-1)}}.\]

We are left to show that our intervals $I$ and $J$ satisfy the following conditions. \\

\begin{enumerate}
    \item $Y(J) > Z(I)$ \\
    
    To see this, we have \begin{align*}
        Y(J)-Z(I) &= \frac{1}{q^{n s}}\left( \frac{k}{p^{ms}}-\frac{j}{q^{m_2(p-1)}}\right)\\
        &= \frac{1}{q^{n s}}\left( \frac{k}{p^{m_1(q-1)}}-\frac{j}{q^{m_2(p-1)}}\right)\\
        &= \frac{1}{q^{ns}}\left( \frac{1}{p^{m_1(q-1)}q^{m_2(p-1)}}\right),
    \end{align*}which is greater than $0$ by construction. \\
    \item $I \subset J$ \\

    Let $l(I)$ denote the left most endpoint of $I$. Then we have
    \begin{align*}
        |[l(I),Y(J)]| &= |[l(I),Z(I)]| + |[Z(I),Y(J)]|\\
        &= \frac{q-1}{q^{ns + m_2(p-1)}}+ \frac{1}{p^{m_1(q-1)}q^{m_2(p-1)}}\\
        &=\frac{1}{q^{ns + m_2(p-1)}}\left[q+\left(\frac{1}{p^{m_1(q-1)}} -1\right)\right]\\
        &<\frac{q}{q^{ns + m_2(p-1)}}\\
        &<\frac{q}{10q^{ns+1}p^{ms}}\\
        &=\frac{1}{10(p^mq^n)^{s}} < |[l(J), Y(J)]|=\frac{1}{(p^mq^n)^{s}}
    \end{align*}
    \item $J$ is the smallest $pq$-adic interval containing $I$.\\
    
    Note that $Y(J)$ is again an interior point of $I$, as 
    \[Y(J)-Z(I) = \frac{1}{q^{n s}}\left( \frac{1}{p^{m_1(q-1)}q^{m_2(p-1)}}\right) \leq \frac{1}{q^{n s +m_2(p-1)}} = |[Z(I), r(I)]|,\]
    where $r(I)$ denotes the right endpoint of $I$. Furthermore, all other $pq$-adic intervals with side-length less than or equal to $|J|$ are either disjoint from $Y(J)$ or contain $Y(J)$ as an endpoint.\\
    
    \item $Y(J)-Z(I)< \varepsilon |I|$\\
    
    This follows from our choice of $m_1$, that is:
    \begin{align*}
        Y(J)-Z(I) &= \frac{1}{q^{ns}}\left( \frac{1}{p^{m_1(q-1)}q^{m_2(p-1)}}\right)\\
        &< \frac{1}{q^{ns}}\left( \frac{\varepsilon q}{q^{m_2(p-1)}}\right)\\
        &= \varepsilon |I|.
    \end{align*}
\end{enumerate}
By the same methodology used in \cite{AH} this completes the proof.
\end{proof}

\subsection{Proof of Theorem \ref{two bases theorem}} In order to prove our main result of this section, we take advantage of the insight that the methods used in \cite{AH} to show the original measure is $p$-adic doubling (i.e. the exhaustion procedure) only require \eqref{key closeness criterion} for \emph{some} $Y$ and $Z$.  That is, that the endpoint of some child of a $p$-adic interval can be made arbitrarily close to some ``distinguished point" on the selected $q$-adic intervals.

In \cite{AH}, these distinguished points are chosen to be $Y(J)$ and $Z(I)$. Given $u,v\in \NN$ with $u>v$, in our extensions we will keep $Z(I)$ as the ``distinguished point" on every $v$-adic interval $I$, but replace $Y(J)$ by a ``revolving" interior endpoint.  As may be expected, this makes the initial framework much more amenable to generalization.\\\

We first generalize the number-theoretic framework developed by \cite{AH} to its full power, and afterwards we will detail the changes which must be made to their interval selection procedure in order to complete our extensions.

\begin{prop} (Generalization of Proposition 2.1 in \cite{AH})
\label{AH generalized order}
Let $p$ be prime and $u = p^k$ for an integer $k \geq 1$, and let $v$ be any integer with $\gcd(u,v)=1$. Let further $O_m(u,v)$ denote the order of $v^{\phi(u)}$ in $\left( \ZZ \big/ \left(u^m \ZZ \right) \right)^*$ for each $m \ge 1$, where $\phi(u)$ denotes the Euler-totient function. Then there exists some integer $C(u,v)\geq 0$, such that
$$
 \frac{O_m(u,v)}{u^{m-1}}=\frac{1}{u^{C(u,v)}}. 
$$
when $m$ is sufficiently large. 
\end{prop}

\begin{proof}
Let $m(u,v)$ be the smallest integer such that
$$
v^{\phi(u)} \not\equiv 1 \ (\textrm{mod} \ u^{m(u,v)+1}). 
$$
This implies that there exists some $N_0 \in \{1, 2, \dots, m(u,v)\}$ such that
\begin{equation} \label{20200726eq01}
\left(v^{\phi(u)} \right)^{u^{N_0}} \equiv 1 \ (\textrm{mod} \ u^{m(u,v)+1}), 
\end{equation}

since by Euler's theorem (applied to $v$), it is always true that
$$
\left(v^{\phi(u)} \right)^{u^{m(u,v)}} \equiv 1 \ (\textrm{mod} \ u^{m(u,v)+1}).
$$  In slightly more detail, recall that if $u=p^k$ then $\phi(u)=p^{k-1}(p-1)$. We see \[u^{N_0}\phi(u)=p^{kN_0+k-1}(p-1),\] and as $p^{km(u,v)+k-1}(p-1)=\phi(p^{km(u,v)+k})$ we see that such an $N_0$ must indeed exist.
Without loss of generality, we assume that the $N_0$ fixed above is the smallest, namely
\begin{equation} \label{20200726eq02}
\left(v^{\phi(u)} \right)^{u^{N_0-1}} \not\equiv 1 \ (\textrm{mod} \ u^{m(u,v)+1}). 
\end{equation}

We claim that for any $\ell \ge 0$, we have
\begin{equation} \label{20200727eq01}
    \left(v^{\phi(u)} \right)^{u^{N_0+\ell-1}} \not\equiv 1 \ (\textrm{mod} \ u^{m(u,v)+\ell+1}). 
\end{equation} 

\medskip

We prove the claim by induction. The case when $\ell=0$ is exactly \eqref{20200726eq02}. Assume \eqref{20200727eq01} holds when $\ell=k$, that is
\begin{equation} \label{20200727eq02}
\left(v^{\phi(u)} \right)^{u^{N_0+k-1}} \not\equiv 1 \ (\textrm{mod} \ u^{m(u,v)+k+1}), 
\end{equation} 
By raising both sides of equation (\eqref{20200726eq01}) to $u^{k-1}$ we have
\begin{equation} \label{20200727eq02'}
\left(v^{\phi(u)} \right)^{u^{N_0+k-1}} \equiv 1 \ (\textrm{mod} \ u^{m(u,v)+k}), 
\end{equation}
which, together with \eqref{20200727eq02}, implies that we can write
$$
\left(v^{\phi(u)} \right)^{u^{N_0+k-1}}=u^{m(u,v)+k} \cdot s+1, 
$$
where $u \nmid s$. 

Taking the $u$-th power on both sides of the above equation, we have
\begin{eqnarray*}
\left(v^{\phi(u)} \right)^{u^{N_0+k}}%
&\equiv& \left(\left(v^{\phi(u)} \right)^{u^{N_0+k-1}} \right)^u \\
&\equiv& \left(u^{m(u,v)+k} \cdot s+1 \right)^u \\
&\equiv& u^{m(u,v)+k+1} \cdot s+1 \\
& \not\equiv& 1 \quad (\textrm{mod} \ u^{m(u,v)+k+2}); 
\end{eqnarray*}
in the last line above, we use the fact that $u \nmid s$.  In particular, we have used that $m(u,v)+k\geq 1$, to show (via the binomial theorem) $2(m(u,v)+k \geq m(u,v)+k +1$. Therefore, \eqref{20200727eq01} is proved.\\\

With this, for any $\ell \geq 0$ we have 
\[(v^{\phi(u)})^{u^{N_0+\ell-1}}\neq 1 (\bmod u^{m(u,v)+\ell +1}),\] 
which implies that 
\[O_{m(u,v)+\ell+1}(u,v)=u^{N_0+\ell}.\] Rewriting $m$ as $m(u,v)+(m-m(u,v)-1)+1=m(u,v)+\ell+1$, as in \cite{AH} this yields 
\[O_m(u,v)=u^{-(m(u,v)-N_0)},\] and setting $C(u,v)=(m(u,v)-N_0)$ gives $u^{-C(u,v)}$, as desired.
\end{proof}

After the next Proposition, we will have the necessary tools to generalize the interval selection procedure, namely, Theorem 3.4 and Proposition 3.5 of \cite{AH}.\\\

\begin{prop}

(Generalization of Proposition 2.3 in \cite{AH})
\label{AH ap} Let $p$ be prime, and $u=p^k$ for an integer $k \geq 1$. Let $v$ be coprime with $C(u,v)$ and $m(u,v)$ be defined as above. Then for any $m_1>\frac{m(u,v)}{\phi(v)}$ and 
\begin{eqnarray} \label{20200727eq04}
k %
&\in& \left\{1, 1+u^{C(u,v)+1}, 1+2u^{C(u,v)+1}, \dots, u^{m_1\phi(v)}-u^{C(u,v)+1}+1 \right\} \nonumber \\
&=& \left\{ x \in \left[1, u^{m_1\phi(v)}\right]: x \equiv 1 \  \left(\textrm{mod} \ u^{C(u,v)+1} \right) \right\}, 
\end{eqnarray}
there exists infinitely many pairs $j$ and $m_2$, where $m_2 \in \NN$, and 
\begin{eqnarray} \label{20200727eq05}
j%
&\in&  \left\{v-1, 2v-1, \dots, v^{m_2\phi(u)}-1 \right\} \nonumber \\
&=& \left\{y \in \left[ 1, v^{m_2\phi(u)} \right]: y \equiv -1 \ (\textrm{mod} \ v) \right\}, 
\end{eqnarray}
such that
\begin{equation} \label{20200727eq06}
\frac{k}{u^{m_1\phi(v)}}-\frac{j}{v^{m_2\phi(u)}}=\frac{1}{u^{m_1\phi(v)}v^{m_2\phi(u)}}.
\end{equation} 
\end{prop}

\begin{proof}
Note that \eqref{20200727eq06} is equivalent to find infinitely many pairs $m_2$ and $j$ which satisfies \eqref{20200727eq05} for the equation 
\begin{equation} \label{20200727eq07}
k v^{m_2\phi(u)}-j u^{m_1\phi(v)}=1, 
\end{equation}
where $m_1> \frac{m(u,v)}{\phi(v)}$ and $k$ satisfies \eqref{20200727eq04}. 

To begin with, we note that if \eqref{20200727eq07} holds, then $j$ automatically satisfies \eqref{20200727eq06}, which follows by taking the modulus $v$ on both sides of \eqref{20200727eq07} and Euler's theorem. 
 
Therefore, it suffices for us to solve \eqref{20200727eq07} for infinitely many pairs $m_2$ and $j$. Taking modulus $u^{m_1\phi(v)}$ on both sides of \eqref{20200727eq07}, we see that it suffices to solve 
\begin{equation} \label{20200727eq08}
    kv^{m_2\phi(u)} \equiv 1 \quad \left(\textrm{mod} \ u^{m_1\phi(v)}\right),
\end{equation}
where 
$$
k \in  \left\{ x \in \left[1, u^{m_1\phi(v)}\right]: x \equiv 1 \  \left(\textrm{mod} \ u^{C(u,v)+1} \right) \right\}.
$$
Denote
$$
G_{m_1}(u,v):=\left\{ x \in \left[1, u^{m_1\phi(v)}\right]: x \equiv 1 \  \left(\textrm{mod} \ u^{C(u,v)+1}\right) \right\}.
$$
The solubility of \eqref{20200727eq08} will again follow if we show that
\begin{enumerate}
    \item [(a)] The set $G_{m_1}(u,v)$ is a subgroup of $\left(\ZZ / \left(u^{m_1\phi(v)} \right) \ZZ \right)^*$; 
    
    \medskip
    
    \item [(b)] $v^{\phi(u)}$ is a generator of the group $ G_{m_1}(u,v)$. 
\end{enumerate}

We now show (a) and (b). To begin, we note that
\begin{equation} \label{20200727eq08}
v^{\phi(u)} \equiv 1 \left(\textrm{mod} \ u^{C(u,v)+1}\right), 
\end{equation}
such that $v^{\phi(u)} \in G_{m_1}(u,v)$. This is because $m(u,v)$ is the smallest integer such that
$$
v^{\phi(u)} \not\equiv 1 \ (\textrm{mod} \ u^{m(u,v)+1})
$$
and $C(u,v)=m(u,v)-N_0$ for some $N_0 \in \{1, \dots, m(u,v)\}$. Next, it follows from \eqref{20200727eq08} that
$$
\left(v^{\phi(u)} \right)^\ell \in G_{m_1}(u,v), \quad \forall \ell \ge 1, 
$$
and hence
$$
\left\langle v^{\phi(u)} \right\rangle \subseteq G_{m_1}(u,v), 
$$.

Therefore, both assertions $(a)$ and $(b)$ will follow if we can show
$$
\left\langle v^{\phi(u)} \right\rangle=G_{m_1}(u,v).
$$
This equality follows from the fact that
$$
O_{m_1\phi(v)}=u^{m_1\phi(v)-C(u,v)-1}= \left|G_{m_1}(u,v) \right|
$$
which again hinges on Proposition \ref{AH generalized order}.
\end{proof}

We can now move on to generalizing the interval selection. Given the previous results, let $\{pc_1,...,pc_k\}$ be a finite set of multiples of some prime $p$, where we have $c_i \in \NN$ for all $i \in [1,k]$. Now, given some $v$ coprime to $p$, we can always find $n\in \NN$ such that $v<p^n=u$, and proceed to build our measure off of $v,u$. Also note that from this point we will work with the finite set of $n^{th}$ multiples $\{(pc_1)^n,...,(pc_k)^n\} = \{uc_1',...,uc_k'\}$ (such that $c_i' = c_i^n$).  As mentioned earlier, this will give us an equivalent result.\\\

\begin{thm}[Generalization of Theorem 3.4 in \cite{AH}]  \label{20200722thm01}
There exists a collection of $v$-adic intervals $\{I^{\alpha_\ell}_\ell \}_{\ell \ge 1}$ on $[0, 1)$, where $\alpha_\ell \ge 1$ is a positive integer associated to $\ell$, such that
\begin{enumerate}
    \item [(1)] The collection of $pc_i$-adic intervals $\{J_{\ell}^i\}_{\ell \ge 1}$ is pairwise disjoint and contained in $[0, 1)$, where $J_{\ell}^i$ is the smallest $pc_i$-adic interval that contains $I_\ell^{\alpha_\ell}$ for all $i \in [1,k]$. In particular, the collection $\{I_\ell^{\alpha_\ell}\}_{\ell \ge 1}$ is also pairwise disjoint;
    
    \medskip

    \item [(2)] For each $\alpha \ge 1, \alpha \in \NN$, there are only finitely many $\ell \ge 1$, such that $\alpha_\ell=\alpha$; 
    
    \medskip
    
    \item [(3)] For each $\ell \ge 1$, 
    \begin{equation} \label{20200722eq03}
    0<\zeta\left(J_{\ell}^i \right)-Z\left(I_\ell^{\alpha_\ell} \right) \le q^{-100\alpha_\ell} \left| I_\ell^{\alpha_\ell} \right|. 
    \end{equation}
    where $\zeta(J_{\ell}^i)$ is an interior endpoint of some child of $J_{\ell}^i$. Note that since $J_{\ell}^i$ is the smallest $pc_i$-adic interval which contains $I_\ell^{\alpha_\ell}$, condition \eqref{20200722eq03} in particular guarantees that the right endpoint of $I_\ell^{\alpha_\ell}$ is to the right of $\zeta(J^\ell)$.
\end{enumerate}
\end{thm}
 \begin{rmk}
        As mentioned above, the measure created via this method will always also be $v$ and $p$-adic doubling, and if we let $c_1=1$ such that $J_{\ell}^1$ is our largest $p$-adic interval containing $I_{\ell}^{\alpha_{\ell}}$, then $\zeta(J_{\ell}^1)$ will always equal $Y(J_{\ell}^1)$.
\end{rmk}

As before this theorem's proof relies on a generalized Proposition 3.5 of \cite{AH} (Proposition \ref{20200805prop01} below); however, while in the original construction the work was essentially done after the proof of Proposition 3.5, we have to take more care.  After the proof of the following proposition we will go into more detail regarding how our intervals are to be selected. \\\

\begin{prop}[Generalized Proposition 3.5] \label{20200805prop01}
Given any interval $\tilde{J} \subset [0, 1]$ and any $\varepsilon>0$, there exists a $v$-adic interval $I \subset \tilde{J}$ such that 
$$
 0<\zeta\left(J^i \right)-Z\left(I \right) \le \varepsilon |I|,
$$
for all integer $i \in [1,k]$, where $J^i$ is the smallest $pc_i$-adic interval that contains $I$. 
\end{prop}

\begin{rmk}
    Note in particular that we do not require $J^i \subseteq \tilde{J}$ (or even $J^i \subseteq [0,1]$) -- however we will necessarily have $I \subset \tilde{J}$ and $I\subset J^i$.
\end{rmk}

\begin{proof}
We start with fixing an interval $\widetilde{J} \subset [0, 1]$ and some $\varepsilon>0$, and we let $J'$ be the largest $v$-adic interval which is contained in $\widetilde{J}$ with sidelength $1/v^{t_1'}$. We choose $t_1> \max\left\{\frac{m(u,v)}{\phi(v)},  t_1' \right\}$ and $\frac{1}{u^{t_1\phi(v)}}<\varepsilon v$, and we choose
$$
k \equiv 1 (\bmod u^{C(u,v)+1}), 
$$
and note that choosing $t_1$ large enough we can find a $k$ such that $\frac{k}{u^{t_1\phi(v)}} \in J'$. Fix such a pair of $t_1$ and $k$, and let 
$$
J:=\left[ \frac{k-1}{u^{t_1\phi(v)}}, \frac{k+u-1}{u^{t_1\phi(v)}} \right].
$$
We then have
$$
Y(J)=\frac{k}{u^{t_1\phi(v)}}
$$
and $J \subseteq J' \subseteq \widetilde{J}$ due to the choice of $t_1$ and the fact $u=p^n>v$. \\\

Note that \[\frac{k}{u^{t_1\phi(v)}} = \frac{kc_i'^{t_1\phi(v)}}{(uc_i')^{t_1\phi(v)}}\] is also the endpoint of some $uc_i'$-adic interval. It follows, writing \[\frac{kc_i'^{t_1\phi(v)}}{(uc_i')^{t_1\phi(v)}}\] in lowest terms, that there must exist some $y \in [1, uc_i'-1]$ such that $\frac{kc_i'^{t_1\phi(v)}}{(uc_i')^{t_1\phi(v)}}$ is the right endpoint of the $y^{th}$ child of some $uc_i'$-adic interval $J^i$, denoted $\zeta(J^i)$ -- for example if $\frac{kc_i'^{t_1\phi(v)}}{(uc_i')^{t_1\phi(v)}}$ is in lowest terms (which occurs when $u, c_i'$
are coprime as $u \nmid k$), then $\zeta(J^i)$ will be the right endpoint of the $y^{th}$ child of \[J^i = \left[ \frac{kc_i'^{t_1\phi(v)} - y}{(uc_i')^{t_1\phi(v)}}, \frac{kc_i'^{t_1\phi(v)} - y + uc_i'}{(uc_i')^{t_1\phi(v)}}\right).\] Henceforth we will let $J^i$ denote the smallest $uc_i'$-adic interval such that $\zeta(J^i)$ is the endpoint of the $y^{th}$ child.\\\

By Proposition \ref{AH ap}, there exists infinitely many pairs $t_2 \in \NN$ and 
$$
j \equiv -1 (\bmod v)
$$
such that
\begin{equation} \label{20200805eq01}
\frac{k}{u^{t_1\phi(v)}}-\frac{j}{v^{t_2\phi(u)}}=\frac{1}{u^{t_1\phi(v)}v^{t_2\phi(u)}}.
\end{equation} 
We choose such a pair $t_2$ and $j$, with $t_2$ sufficiently large such that
\begin{equation} \label{20200805eq02}
\frac{v-1}{v^{t_2\phi(u)}}+\frac{1}{u^{t_1\phi(v)}v^{t_2\phi(u)}}< \inf_{i \in [1,k]} \left \{\frac{1}{(uc_i')^{t_1\phi(v)}}\right \}
\end{equation} 
and let
$$
I:=\left[\frac{j+1-v}{v^{t_2\phi(u)}}, \frac{j+1}{v^{t_2\phi(u)}} \right].
$$
such that 
$$
\textZeta(I)=\frac{j}{v^{t_2\phi(u)}}.
$$.\\\

We now proceed to check items (1) through (4) in the proof of Proposition 3.5 of \cite{AH}, and doing so complete the proof.
        \begin{enumerate}
            \item The inequality $\zeta(J^i)>Z(I)$ follows because $\zeta(J^i)=Y(J)>Z(I)$ for all $i \in [i,k]$.\\\

            \item To show $I \subset J^i$, consider \begin{align*}
                |[l(I), \zeta(J^i)]| &= |[l(I), Z(I)]|+|[Z(I), \zeta(J^i)]|\\
                &= \frac{v-1}{v^{t_2\phi(u)}}+ \frac{1}{u^{t_1\phi(v)}v^{t_2(c-1)}}\\ &<\frac{1}{(uc_i')^{t_1\phi(v)}} \\&\leq |[l(J^i), \zeta(J^i)]|\end{align*} by our choice of $t_2$. Explicitly, this tells us that $l(I)>l(J^i)$ and $Z(J^i)\geq \zeta(J^i) > Z(I)$ which gives $r(J^i)>r(I)$, so we do have that $I \subset J^i$ for all $i$.\\\

            \item As before $I$ contains $\zeta(J^i)$ as an interior point (this is obvious considering just $I, J$), so indeed $J^i$ must be the smallest $uc_i'$-adic interval containing $I$.\\\

            \item The inequality $\zeta(J^i)-Z(I)<\varepsilon |I|$ again follows exactly as before by our choice of $t_1$, as \[\zeta(J^i)-Z(I) = \frac{k}{u^{t_1\phi(v)}}-\frac{j}{v^{t_2\phi(u)}}=\frac{1}{u^{t_1\phi(v)}v^{t_2\phi(u)}} < \frac{\varepsilon v}{v^{t_2\phi(u)}}  = \varepsilon|I|.\]
        \end{enumerate}
\end{proof}

With the proof complete, we can detail the additional complications that arise when applying the above result to prove Theorem \ref{20200722thm01}.\\\ 

To begin, consider $uc_i'$ such that $u, c_i$ are coprime, where \[J^i = \left[ \frac{kc_i'^{t_1\phi(v)} - y}{(uc_i')^{t_1\phi(v)}}, \frac{kc_i'^{t_1\phi(v)} - y + uc_i'}{(uc_i')^{t_1\phi(v)}}\right).\] Recalling that $\zeta(J^i)=Y(J)$, we see that it is possible that such a $J^i$ could possibly extend to the left of $l(J)$ (for example this may happen if $\zeta(J^i)=Z(J^i)$).\\\

In the application of Proposition 3.5 to Theorem 3.4 in \cite{AH}, the authors began by choosing $p$-adic intervals, and then selecting $J_{\ell}$ and $I^{\alpha_{\ell}}_{\ell}$ within each of these intervals. We also want to begin by choosing $p^n$-adic intervals, but if we choose our intervals carelessly (not paying attention to the initial $\tilde{J}$ in which our interval selection begins) it seems possible that one of our $J^i_{\ell}$ could intersect some $J^j_{\ell'}$ and the corresponding $I_{\ell'}$, which would be problematic.\\\

To avoid this, we note that for $u, c_i'$ coprime $J^i$ will always have total length less than $J$; thus it suffices to choose $u$-adic intervals $J$ at least twice their length apart. In other words, if we choose to find $J_{\ell}$ and $I_{\ell}$ within a $u$-adic $\tilde{J}$ of length $u^{-d}$, we would want our next $\tilde{J'}$ to be at least $u^{-d}$ units right of $r(\tilde{J})$. As we can make all of our intervals arbitrarily small, this can clearly be accomplished (an alternate way to look at it is that, given our largest $\tilde{J}$ has length $u^{-d}$, $d$ can clearly be chosen such that $2u^{-d}+2u^{-2d}+2u^{-3d}+\cdots=\frac{2}{u^d-1}$ becomes arbitrarily small) 

Now we deal with the case where $u, c_i'$ are not coprime. As alluded to earlier, the trouble here is that we a priori don't have any bounds on the size of $J^i$ chosen within an arbitrary $\tilde{J}$.  However, if we choose an $uc_i'$-adic point contained strictly inside an $uc_i'$-adic interval $\tilde{J}_i$, then this $uc_i'$-adic must be the endpoint of some descendant of $\tilde{J}_i$.\\\

With this insight, we choose our intervals as follows: 
\begin{enumerate}
    \item  choose a sequence of $u$-adic intervals $\tilde{J_{\ell}}$, spaced apart as described above
    \item begin with the least $c_i'$ not coprime to $u$, and choose some $uc_i'$-adic interval contained within each $\tilde{J_{\ell}}$ (for all $\ell$) -- denote this interval $\tilde{J^i_{\ell}}$
    \item repeat this process for the next largest $c_j'$ not coprime to $u$, choosing some $uc_j'$-adic interval contained within each $\tilde{J_{\ell}^i}$ denoted $\tilde{J_{\ell}^j}$
    \item iterate steps (2) and (3) until there are no more $c_i'$ which are not coprime to $u$
    \item let $c_w'$  denote the largest $c_i'$ not coprime to $u$, such that we end up with infinitely many disjoint $c_w'$-adic intervals $\tilde{J_{\ell}^w}$ 
\end{enumerate}
Choosing $t_1$ large enough, we can always find $Y(J_{\ell})=\zeta(J^i_{\ell})$ contained within each of our $\tilde{J_{\ell}^w}$, and by the reasoning given above, all of our $uc_i'$-adic intervals $J^i_{\ell} $ will be disjoint, and each will be the largest $uc_i'$-adic interval containing our $v$-adic interval $I_{\ell}$, as desired.\\\

From the comments in the introduction in \cite{AH} and from a careful reading of those results, we verify that we have generalized all steps needed to prove Theorem \ref{two bases theorem}.  Indeed, the construction of the measure by re-weighting the chosen intervals is the same as \cite{AH} (and as described above) once we have selected our intervals $I^{(\alpha)}$.  Therefore, the proofs that this measure constructed on the chosen $I^{(\alpha)}$ is not doubling but is $u$ and $v$-adic doubling remain carry through without change.\\\

\section{A new construction for any finite list of bases: interplay of number theory, topology, and analysis}
\label{finite bases}
We begin by giving some intuition as to why the construction given above fails to admit a generalization to arbitrary finite lists of natural numbers. While details are outlined in the appendix, essentially it relies on the properties of the \emph{far-numbers} as discussed in \cite{AH} and \cite{AH2}.
\begin{defn} (Definition 2.7 in \cite{AH})
    A real number $\del$ is \emph{$n$-far} if the distance from $\del$ to each given rational of the form $k/n^m$ is at least some fixed multiple of $1/n^m$, where $m \geq 0$, $k\in \ZZ$. That is, if there exists $C>0$  such that
\begin{equation}
\label{C delta}
   \left| \del-\frac{k}{n^m} \right| \ge \frac{C}{n^m}, \quad \forall m \ge 0, k \in \ZZ.
\end{equation}
where $C$ depends on $\del$ but is independent of $m$ and $k$. 
\end{defn}
With this definition in mind, we now attempt to extend the framework of \cite{AH} in the setting of primes $p$ and $q$, to the setting of three distinct primes.  Naively, given primes $p, q,$ and $r$ one might try to extend the construction in \cite{AH} by finding a $p$-adic interval $J_p$ and a $q$-adic interval $J_q$ which are the largest $p$ and $q$-adic intervals (respectively) containing some $r$-adic interval $I_r$ and simultaneously satisfy \[Y(J_p)-Z(I_r) \leq \varepsilon |I_r|\] and \[Y(J_q)-Z(I_r) \leq \varepsilon |I_r|\] for some given arbitrarily small $\varepsilon>0$.
Observe that all $p$-adic points contained in $(0,1)$ are $q$-far, and similarly all $q$-adic points contained in $(0,1)$ are $p$-far. More concretely, if we have $Y(J_p) = \frac{k}{p^n}$ and $Y(J_q)=\frac{j}{q^m}$ for some $k,j,m,n \in \NN$, then it can be easily verified that $|Y(J_p)-Y(J_q)|\geq \frac{1}{p^nq^m}$ (see \cite{AHJOW}). The triangle inequality then yields \[|Y(J_p)-Y(J_q)|\leq |Y(J_p)-Z(I_r)|+|Y(J_q)-Z(I_r)|\leq 2\varepsilon |I_r|,\] such that $\frac{1}{p^nq^m} \leq 2\varepsilon |I_r|$ for any arbitrarily small $\varepsilon$. This would imply that at least one of $\frac{1}{q^m} = \frac{|J_q|}{q}$ or $\frac{1}{p^n} = \frac{|J_p|}{p}$ was of order much smaller than $|I_r| = \frac{1}{r^k}$ for some $k \in \NN$. Recalling that we would want $J_p, J_q$ to both contain $I_r$ this seems inherently problematic, and motivates us to try and find a new approach in order to select intervals which are sufficiently ``close" for multiple primes. This new approach is outlined throughout the rest of this section.

\begin{thm}
    Let $\{n_1,\ldots,n_k\}\subseteq\NN$ be a finite set of natural numbers which are not powers of $2.$ Then, there exists an infinite family of measures which are $n_i$-adic doubling for all $i$ but not doubling overall.
\end{thm}

Note that for technical reasons, we will always exclude the set of bases that only consists of $1\in\NN$, otherwise we would tautologically claim the existence of a measure that is both non-doubling and doubling.
\begin{rmk}
    It suffices to assume that no $n$ is a power of $2$ because we will construct a dyadic doubling measure, which is equivalent to a $2^n$-adic doubling measure for each $n\in\NN.$  We choose to make it dyadic doubling only out of convenience to guarantee the proper containment relations from the intervals that will be constructed..
\end{rmk}

By a careful reading of \cite{AH}, in particular, Sections $5$ and $6$ (and as discussed in Section \ref{two bases} here), it suffices to find an infinite family of disjoint intervals $I^{(\alpha)}$ (with $\alpha\in\NN$), where each contains a dyadic interval $I_2^{(\alpha)}$ satisfying \eqref{key closeness criterion}.  Namely, for all $1\leq i \leq m$ we want that \[|Y(I_{n_i}^{(\alpha)})-Z(I_2^{(\alpha)})|<2^{-100\alpha}|I_2^{(\alpha)}|\] where $I_{n_i}^{(\alpha)}$ is the smallest $n_i$-adic interval in $I^{(\alpha)}$ containing $I_2^{(\alpha)}.$
We will construct our measure through dyadic intervals and apply the weighting procedure from Section $4$ of \cite{AH} (and outlined in Section \ref{two bases} here) to these intervals.  Then, we get a measure which is dyadic doubling and $n_i$-adic doubling but not doubling overall.  However, as related above, guaranteeing the closeness relationship is problematic.  Therefore, we will permit a key relaxation of \eqref{key closeness criterion} by allowing \[|Y(I_{n_i}^{(\alpha)})-Z(I_2^{(\alpha)})|<2^{-2\alpha}|I_2^{(\alpha)}|\] rather than trying to get \[0<Y(I_{n_i}^{(\alpha)})-Z(I_2^{(\alpha)})<2^{-2\alpha}|I_2^{(\alpha)}|.\] As discussed in Section $5$ and $6$ in \cite{AH}, this bound on the absolute value is all that is really needed.  Indeed, the same proof holds just with us having to switch directions in running through the outlined exhaustion procedure; the authors in \cite{AH} choose the specific interior points to be in the given orientation for number theoretic purposes.  Our replacement for this stems from the closeness relationship below.

\begin{lemma}
\label{finite base approx}
    Let $\{n_1,\ldots,n_k\}\subseteq\NN,$ $\alpha\in\NN$ and $\varepsilon = 2^{-100\alpha}>0.$ Then, there exists infinitely many $x\in \NN$ such that for all $1\leq i \leq k$ that \begin{equation} \label{above} \left|\frac{1}{2^x}-\frac{1}{n_i^{\left[x\log_{n_i}(2)\right]}}\right|<\varepsilon \, \frac{1}{2^x}.\end{equation}
\end{lemma}
    The significance of the $\alpha$ here is that in the measure construction we will choose our $x$ such that $\varepsilon =2^{-100\alpha}<< 2^{-2\alpha}$ just as done in \cite{AH}.
\begin{rmk}
     Upon showing the existence of such an $x$ we can construct dyadic intervals $I_2^\alpha$ in $[\alpha,\alpha+1]$ such that if $J_i^\alpha$ is the smallest $n_i$-adic interval containing $I_2^\alpha$ then $|Y(J_i^\alpha)-Z(I_2^\alpha)|<\varepsilon|I_2^\alpha|$ which is sufficient to guarantee the existence of a measure which is dyadic doubling and $n_i$-adic doubling for each $n_i$. Alternatively, we could construct this measure on any $n$-adic intervals where $n$ is at most the smallest of our list of natural numbers with only minor modifications in the proofs below. 
     \end{rmk} 

\begin{proof}
    Re-arranging Equation \ref{above} is equivalent to finding values of $x$ so that \[\left|1-\frac{2^{x}}{n_i^{\left[x\log_{n_i}(2)\right]}}\right|<\varepsilon\] 
    
    for all $i$. Note that the orbit of rationally independent points $(\theta_1,\ldots,\theta_m)\in\TT^m$ under multiplication by an integer is dense in $\TT^m.$ It is an open problem as to whether \[S=\{\log_{n_i}(2)| 1\leq i \leq k\}\] is a rationally independent set -- this is equivalent to asking whether $\{\frac{1}{\log n_i}| 1 \leq i \leq k\}$ is rationally independent, which the Schaunel conjecture (if true) would imply. Hence we will consider two cases without casting any aspersions on the validity of either case.

    If $S$ is rationally independent, then the orbit of these points is dense in $\TT^k.$
If the set $S$ is rationally dependent, then the orbit under the $\ZZ$-action will be dense in some linear subspace of $\TT^k$ (as proven below in Proposition \ref{dense in torus}). 
    
    Thus in either case, for $\frac{1}{2}>\delta>0$, there exist infinitely many $x$ such that $|\{x\log_{n_i} 2\}|\leq \delta$  and given $\varepsilon$, we can pick $\delta$ small enough such that $|1-n_i^{\pm\delta}|<\varepsilon.$ For any such $x$ we then get that \[\left|1-\frac{2^{x}}{n_i^{\left[ x\log_{n_i} 2\right]}}\right|\leq \left|1-n_i^{\pm\delta}\right|< \varepsilon. \qedhere\]
\end{proof}

\begin{prop}
\label{dense in torus}
    Let $\{\theta_1,\ldots,\theta_k\}\subseteq \RR\setminus \QQ$ be rationally independent. Then, the orbit of $(\theta_1,\ldots,\theta_k)$ under the action of $\ZZ$ is dense in some linear subspace of $\TT^k.$
\end{prop}
\begin{proof}
    Because $\theta_1,\ldots,\theta_k$ are rationally independent there is some minimal $\QQ$-spanning set. Without loss of generality say \[\{\theta_1,\ldots,\theta_k\}\subseteq \Span_\QQ\{\theta_1,\ldots,\theta_j\}\] with $j<k$ and $\theta_1,\ldots,\theta_j$ rationally independent. Consider the linear subspace \[V:=\Span_\RR \left\{\begin{pmatrix}\theta_1\\0\\0\\\vdots\\0\\\theta_{j+1}\\\vdots\\\theta_k\end{pmatrix},\begin{pmatrix}0\\\theta_2\\0\\\vdots\\0\\\theta_{j+1}\\\vdots\\\theta_k\end{pmatrix},\ldots,\begin{pmatrix}0\\0\\0\\\vdots\\\theta_j\\\theta_{j+1}\\\vdots\\\theta_k\end{pmatrix},\right\}/\ZZ^k\subseteq \TT^k.\]

    We aim to show that the orbit of $(\theta_1,\ldots,\theta_k)$ is dense in this $j$-dimensional subspace $V\subseteq \TT^k.$ Take $(x_1,\ldots,x_k)\in V$ and $\varepsilon>0.$ Since $(\theta_1,\ldots,\theta_j)$ is rationally independent, the $\ZZ$-orbit is dense in $\TT^j.$ Hence, there is some $x\in \ZZ$ such that for all $1\leq i \leq j$ we have $|x_i-x\theta_i|<\frac{\varepsilon}{C}$ where 
    \[C=\max_{j+1\leq \ell \leq k}\left\{\sum_{i=1}^j |r_i|:\theta_{\ell}=r_1\theta_1+\cdots+r_j\theta_j\right\}.\] 
    
    This constant is well defined because $\theta_1,\ldots,\theta_j$ are rationally independent, so for each $\theta_{\ell}$ with $\ell>j$, there are unique $r_1,\ldots,r_j$ such that $\theta_{\ell}=r_1\theta_1+\cdots+r_j\theta_j.$
    Then, for all $j+1\leq \ell \leq k$ we get that \[x_\ell = r_1x_1+\cdots+r_jx_j\] and \[\theta_\ell=r_1\theta_1+\cdots+r_j\theta_j.\] This gives
    \[|x_\ell-x\theta_\ell|=\sum_{i=1}^{j} |r_i||x_i-x\theta_i|<\frac{\varepsilon}{C}(\sum_{i=1}^{j} |r_i|)<\varepsilon.\] Hence, for each $(x_1,\ldots,x_k)\in V$ we have an $x\in\NN$ such that \[\|(x\theta_1,\ldots,x\theta_k)-(x_1,\ldots,x_k)\|_\infty<\varepsilon;\] thus the $\ZZ$-orbit of $(\theta_1,\ldots,\theta_k)$ is dense in $\TT^k.$ 
\end{proof}

\begin{rmk}
    If the Schanuel conjecture \cite{Sch} were true then $\{\log_{n_i}(2)|1\leq i\leq k\}$ would be rationally independent and thus we could guarantee that the orbit would be dense in the torus. Hence, we would be able to choose points in the orbit in any neighborhood of $0.$ In particular we could guarantee that $\{x\log_{n_i}(2)\}<\frac{1}{2}$ and arbitrarily close to $0.$ we could then choose $x$ such that $[x\log_{n_i} 2]=\lfloor x\log_{n_i}(2)\rfloor$ giving \[\frac{1}{2^x}<\frac{1}{n_i^{[x\log_{n_i}(2)]}}\] for all $n_i$. this would guarantee that the intervals we construct would have $Z(I_2^{(\alpha)})$ would be to the left of the $Y(I_{n_i}^{(\alpha)})$ for each $n_i$ which would mimic the construction in \cite{AH}. This alignment would give a nice interpretation as to why the original number theoretic construction of \cite{AH} does not carry over in this setting.  Explicit details related to this are given in the appendix.  The relationship to the Schanuel conjecture connects many areas in our setting.
\end{rmk}
Now, we prove Theorem \ref{finite bases theorem}.
\begin{proof}[Proof of Theorem \ref{finite bases theorem}]
    We begin by selecting the collection of intervals on which to alter our measure. For each $\alpha\in\NN$ find a requisite $x$ as constructed Lemma \ref{finite base approx} for $\varepsilon:=2^{-100\alpha}.$
    Then, define \[I_2^{(\alpha)}=\left[\alpha,\alpha+\frac{1}{2^{x-1}}\right)\] and \[I_{n_i}^{(\alpha)}=\left[\alpha,\alpha+\frac{1}{n_i^{[x\log_{n_i}(2)]-1}}\right).\]
    This gives
    \[|Y(I_{n_i}^{(\alpha)})-Z(I_2^{(\alpha)})|\leq \varepsilon|I_2^{(\alpha)}|.\] Hence it suffices to show that for all $n_i$ we have the containment $I_2^{(\alpha)}\subseteq I_{n_i}^{(\alpha)}$ as the proximity of $Y(I_{n_i}^{(\alpha)})$ and $Z(I_2^{(\alpha)})$ guarantees that this is the smallest such interval containing it.  We enlose a picture of what nested $n$-adic intervals would look like for $n=2,3,5,6$ in Figure \ref{finite base figure}. Set the following notation: for an interval $I=[a, b) \subset \RR$, let $l(I)=a$ and $r(I)=b$ denote the left and right-most endpoints of the closure of $I$, respectively. Containment follows from looking at the relative position of the endpoints. Indeed, $l(I_2^{(\alpha)})=l(I_{n_i}^{(\alpha)})$ and we will show that  \[r(I_{n_i}^{(\alpha)})=\alpha+n_i y\]  is to the right of \[r(I_2^{(\alpha)})=\alpha+2z\] where $y=\frac{1}{n_i^{[x\log_{n_i}(2)]}}$ and $z=\frac{1}{2^x}.$ Because $n_i>2$ we have that \[r(I_{n_i}^{(\alpha)})-r(I_2^{(\alpha)})=(n_i-2)y+2(y-z)>0\] as $2(y-z)>-2\varepsilon z$ and $\varepsilon$ is much smaller than $n_i-2 \geq 1$.

    As discussed earlier, one can follow the procedure in Sections $4,5$ and $6$ of \cite{AH} to get the main result.  In particular, the proof of $n$-adic doubling for each $n$ is essentially the same as in \cite{AH} but with some of the details slightly changed. First off, the ``trivial" cases of finding the ratio for $\frac{\mu(J_{j_1})}{\mu(J_{j_2})}$ for children of the $n$-adic interval $J$ containing $I_2^{(\alpha)}$ are the exact same in this case.

    In the other cases, the proof is slightly different. This is because $Y(I_n^{(\alpha)})$ may be to the right or the left of $Z(I_2^{(\alpha)}).$ We run through these cases.

    Suppose that $J=I_n^{(\alpha)}$ for some $n\in\NN$ and large enough $\alpha.$ Let $J_1,\ldots,J_n$ denote the $n$-adic children of $J.$ First, suppose that $Y(J)>Z(I)$ where $I=I_2^{(\alpha)}.$ In that case we get that \[\frac{\mu(J_{j_1})}{\mu(J_{j_2})}=1\;\forall j_1,j_2\in \{3,\ldots,n\}\] as on those intervals we'll just get the Lebesgue measure as they are outside of $I.$ 

    Then, for $J_i$ for $i=1,2$ we get \[\frac{a|I|}{2}\leq \mu(J_1)\leq |I|\] and \[\frac{b|I|}{4}\leq \mu(J_2)\leq 2|I|.\]

    In the case where $Y(J)<Z(I)$ we get \[\frac{\mu(J_{j_1})}{\mu(J_{j_2})}=1\;\forall j_1,j_2\in \{4,\ldots,n\}\] as on those intervals we'll just get the Lebesgue measure as they are outside of $I.$ Then, for $J_i$ for $i=1,2,3$ we get \[\frac{a|I|}{4}\leq\mu(J_1)\leq \frac{a|I|}{2}\] and \[\frac{b|I|}{4}\leq \mu(J_2)\leq |I|\] and \[\frac{a|I|}{4}\leq\mu(J_3)\leq \frac{3|I|}{2}.\]

    The case discussed above is the only one that is essentially different (and only in a small way).  When it comes to the exhaustion procedure most of the details are then identical with the exception of replacing $p$ by $n$.  Thus there exists a family of measures which is $n_i$-adic doubling for each $n_i$ but not doubling overall.
\end{proof}

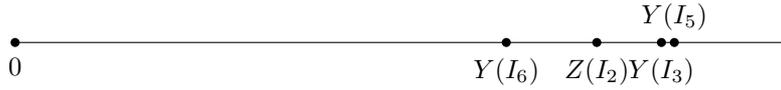
\begin{figure}[ht]
\begin{tikzpicture}[scale=8.5]
\label{finite base figure}
\draw (0,0) -- (1.2,0); 
\fill (0,0) circle [radius=.2pt];
\fill (1,0) circle [radius=.2pt];
\fill (0, -0.01) node [below] {$0$};
\fill (1, -0.01) node [below] {$Y(I_3)$};
\fill (.9, 0) circle [radius=.2pt];
\fill (.9, -0.01) node [below] {$Z(I_2)$};
\fill (1.02, 0) circle [radius=.2pt];
\fill (1.02, .075) node [below] {$Y(I_5)$};
\fill (.76, -0.01) node [below] {$Y(I_6)$};
\fill (.76, 0) circle [radius=.2pt];

\end{tikzpicture}
\caption{$I_2,I_3,\ldots,I_n$.  Not to scale.}
\end{figure}

Now, with an eye towards applications, we will compactify the construction of this measure. That way we will get applications to weight and function classes over compact sets.

Here is the adjusted procedure. Given an initial $\alpha$ take an $x$ such that \[\left|\frac{1}{2^{x}}-\frac{1}{n^{[x\log_n 2]}}\right|<\frac{2^{-100\alpha}}{2^x}.\]

Then, we choose 
\[I_2^{(\alpha)}=\left[0,\frac{1}{2^{x-1}}\right) \text{ and } I_n^{(\alpha)}=\left[0,\frac{1}{n^{[x\log_n 2]-1}}\right).\] We keep the same re-weighting procedure on $I_2^{(\alpha)}$ but with the exception of re-weighting the first step as follows: give the interval $[0,\frac{1}{2^{x+1}}]$ the weight $1$, the interval $[\frac{1}{2^{x+1}},\frac{1}{2^x}]$ with the weight $a$, and $[\frac{1}{2^x},\frac{1}{2^{x-1}}]$ the weight $b-1.$ Then for each successive iteration go back to the original re-weighting procedure of multiplying by $b$ on the right half and $a$ on the left half. Then for any successive $\alpha$ take $x$ such that \[\left|\frac{1}{2^{x}}-\frac{1}{n^{[x\log_n 2]}}\right|<\frac{2^{-100\alpha}}{2^x},\] with $x$ sufficiently larger than the $x$ associated to the previous $\alpha$. This can be done since there are infinitely many $x$ of this form associated to this $\alpha.$ This re-weighting still provides us with a measure as we still get 

\begin{align*}
\frac{1}{2^{x-1}}&=\mu([0,\frac{1}{2^{x-1}}])\\
&=\mu([0,\frac{1}{2^{x+1}}])+\mu([\frac{1}{2^{x+1}},\frac{1}{2^x}])+\mu([\frac{1}{2^{x}},\frac{1}{2^{x+1}}])\\
&=1(\frac{1}{2^{x+1}})+(a)(\frac{1}{2^{x+1}})+(b-1)(\frac{1}{2^x})=\frac{1}{2^{x-1}}.
\end{align*}

Repeating the same procedure then compactifies the measure to $[0,1]$, which is $n$-adic doubling for all $n$ in our list but not doubling overall. 
   
\subsection{Obstacles to extending to an infinite set of bases}
It remains an open question to construct a measure that is doubling on an infinite set of bases, yet not doubling overall.  Here we show that while parts of our construction generalize to the infinite base setting, there are still significant obstacles to overcome to prove the folkloric conjecture.  In particular, in attempting to solve this question we run into a key obstacle. This obstacle turns out to be more general and we shed insight into what a construction to overcome this would need to surmount. Nevertheless, we show that at least our interval selection generalizes to the case of infinitely many $n\in\NN.$ That is, we can find some $x\in\NN$ such that \[\left|\frac{1}{2^x}-\frac{1}{n^{[x\log_n 2]}}\right|<\frac{\varepsilon}{2^x}\] for all $n \in\NN$.

We begin by using the fact that the infinite torus $\TT^\infty$ is metrizable, and thus is endowed with the triangle inequality, to extend the interval selection above.  This would accommodate an infinite base set.  Though standard, for completeness we show the metrizability of $\TT^\infty$ in the appendix. 
\begin{lemma}
    Let $\varepsilon>0.$ Then, there exists infinitely many $x\in\NN$ such that for all $n\in\NN$ which is not a power of $2$ we have \[\left|\frac{1}{2^x}-\frac{1}{n^{[x\log_n 2]}}\right|<\frac{\varepsilon}{2^x}.\]
\end{lemma}
\begin{proof}
    Consider \[z=(\log_3(2),\log_5(2),\log_6(2),\ldots)\in \TT^\omega\] where $z_i$ is $\log_{n_i}(2)$ where $n_i$ is the $i^{\text{th}}$ natural number which is not a power of $2.$ By Proposition \ref{torus 1} we have that in any finite sub-torus $0$ is a limit point of the $\ZZ$-orbit of $z.$ Hence, by Proposition \ref{torus 2} we have that $0$ is a limit point of the $\ZZ$-orbit of $z$ in $\TT^\omega.$ Hence, there exists infinitely many $x\in \NN$ such that $|\{x\log_n 2\}|<\delta<\frac{1}{2}$ for each $n$ such that \[\left|1-\frac{2^x}{n^{[x\log_n 2]}}\right|=|1-n^{\delta}|<\varepsilon\] so \[\left|\frac{1}{2^x}-\frac{1}{n^{[x\log_n 2]}}\right|<\frac{\varepsilon}{2^x}.\] In the case where $n$ is a power of $2$ then there are infinitely many $x$ such that we get $0.$
\end{proof}

\begin{rmk}
While the interval construction works for an infinite list, the issue that remains is the proof of $n$-adic doubling. This is because the constants chosen to ensure doubling cannot be made uniform under the given exhaustion procedure in \cite{AH}. In particular, all of the constants for $n$-adic doubling depend on $N=\lfloor \log_2 n\rfloor +1$ which becomes arbitrarily large, so we cannot ``exhaust" all $n$-adic children (in the language of \cite{AH}) for infinitely many $n$ uniformly.  This indicates a logarithmic growth in the doubling constant, depdendent on the base $n$.  Making this constant uniform does not seem possible under this construction.
\end{rmk}
\subsection{Relationship to Hasse principle}
Finally, we comment about an analogy of our progress toward a Hasse principle. In general terms, the Hasse principle states that if a system of equations has a solution $\mathrm{mod}\, p$ for all primes $p$ and a real solution (which can be viewed as $\mathrm{mod}\, p_\infty$, or the prime at infinity), then it has an integer solution. In \cite{AH}, we introduced an analogue of the Hasse principle with the ``solution $\mathrm{mod}\, p$" being replaced by ``$p$-adic doubling" and an ``integer solution" being replaced by ``doubling". Thus, being able to construct a measure that is doubling for all primes yet not overall would be an analogue of the Hasse principle in this harmonic analysis setting. Here we amplify this analogy by additionally connecting the existence of a real solution to our setting. In our setup, the real solution corresponds to the measure being doubling on any interval around $0$. Since $0$ is never an interior point for any $n$-adic interval, capturing the behavior of intervals containing $0$ cannot be done by the $n$-adic framework and needs to be considered separately, just as real solutions do for the Hasse principle. Indeed, the failure of doubling for our examples stems from intervals not including $0$, so all of our measures on $\RR$ satisfy doubling ``at $0$". Interestingly, in the next section we show that for the function class $VMO$, the infinite intersection of $n$-adic VMO is never equal to the full VMO; however in this example, the failure occurs around the point $0$ (and thus this example would not show the failure of the Hasse principle, since there would be no ``real solution"). Either finding a new construction or showing one cannot exist to demonstrate or break the Hasse principle in this setting would be an exciting development. 

\section{Applications}
\label{applications}
We are now ready to show a variety of applications to weight and function classes. We begin with definitions and short remarks, segway into an observation about the VMO function class involving infinite intersections, and closing with the proof of Theorem \ref{application theorem}. 

Let $w_{\mu}$ be the weight associated to the measure $\mu$ in Section 4, that is

$$
\mu(I) = \int_{I}w_{\mu}dx\text{, for any interval $I$}. 
$$

\begin{defn}\label{RH_nweightsdefinition}
    Let $1 < r < \infty$. We say $w \in RH_r$, the \emph{r-reverse H\"older class}, if 
    $$
    (\fint_I w^r)^{\frac{1}{r}} \leq C \fint_I w
    $$
    for all intervals $I$, where $C$ is an absolute constant. Moreover, we say $w \in RH_1$ if $w \in RH_r$ for some $r>1$, that is 
$$
RH_1 = \bigcup_{r > 1}RH_r.
$$
\end{defn}

\begin{defn}\label{RH_n^pweightsdefinition}
    Let $1 < r < \infty$. We say $w \in RH_r^p$ if 
    $$
    (\fint_P w^r)^{\frac{1}{r}} \leq C \fint_P w
    $$
    for all $p$-adic intervals $P$, where $C$ is an absolute constant and $w$ is $p$-adic doubling.
\end{defn}

\begin{defn}\label{A_rweightsdefinitions}
    Let $1 < r < \infty$. We say a weight $w \in A_r$, the \emph{Muckenhoupt $A_r$ class} if
    $$
    \sup_I(\fint_I w(x) dx)(\fint_I w(x)^{\frac{-1}{r-1}}dx)^{r-1} < \infty
    $$
    where the supremum is taken over all intervals $I$. Moreover, we say $w \in A_{\infty}$ if $w \in A_r$ for some $r > 1$, that is 
    $$
    A_\infty = \bigcup_{r>1}A_r.
    $$
\end{defn}

\begin{defn}\label{A_r^pweightsdefinitions}
    Let $1 < r < \infty$. We say a weight $w \in A_r^p$ if
    $$
    \sup_P(\fint_P w(x) dx)(\fint_P w(x)^{\frac{-1}{r-1}}dx)^{r-1} < \infty
    $$
    where the supremum is taken over all $p$-adic intervals $P$. Moreover, we say $w \in A_{\infty}$ if $w \in A_r^p$ for some $r > 1$, that is 
    $$
    A_\infty^p = \bigcup_{r>1}A_r^p.
    $$
\end{defn}

\begin{defn}\label{BMOdefinition}
    We say a function is in $BMO$ if and only if \begin{align*}
        ||f||_{BMO} \coloneqq \sup_I \fint_I |f - \fint_I f|< \infty 
    \end{align*}
    where $I$ is any interval and $\fint_I$ gives the average over $I$ with respect to Lebesgue measure. A function is similarly in $BMO_p$ if and only if it satisfies the above condition when restricted to $p$-adic intervals $I$.
\end{defn}

The class $VMO$, or functions of vanishing mean oscillation, was introduced by Sarason on the torus \cite{S}.  However, Coifman and Weiss \cite{CW} modified Saranson's definition when they extended it to the real line so as to have a duality relationship with the Hardy space $H^1$ (see discussion below).  We list the definition on the real line first.
\begin{defn}\label{VMOdefinition}
    A function $f$ in $BMO(\RR)$ is said to be in $VMO(\RR)$ if  it satisfies the following conditions:
    \begin{enumerate}
        \item $\lim_{\delta \rightarrow 0} \sup_{I: |I|< \delta}\fint_I |f-\fint_If| = 0$;
        \item $\lim_{N \rightarrow \infty} \sup_{I: |I|> N}\fint_I |f-\fint_If| = 0$; and 
        \item $\lim_{R \rightarrow \infty} \sup_{I: I\cap B(0,R) = \emptyset}\fint_I |f-\fint_I f| = 0$.
    \end{enumerate}
    where $B(0,R)$ is the ball centered at zero of radius $R$.
\end{defn}
With this definition, $VMO$ is the closure in the $BMO$ norm of the class $C_0$ (continuous functions with compact support).
The definition on the torus is the same, but without the second and third conditions.
\begin{defn}\label{VMOpdefinition}
    A function $f$ in $BMO_p (\RR)$ is said to be in $VMO_p(\RR)$ if  it satisfies the following conditions: 
    \begin{enumerate}
        \item $\lim_{\delta \rightarrow 0} \sup_{I: |I|< \delta}\fint_I |f-\fint_If| = 0$;
        \item $\lim_{N \rightarrow \infty} \sup_{I: |I|> N}\fint_I |f-\fint_If| = 0$; and 
        \item $\lim_{R \rightarrow \infty} \sup_{I: I\cap B(0,R) = \emptyset}\fint_I |f-\fint_I f| = 0$,
    \end{enumerate}
    where $I$ is always a $p$-adic interval.
\end{defn}

We will also work with the Hardy sapces $H^1$ and $n$-adic Hardy spaces $H^1_n$.  The classical Hardy space $H^1$ is a subspace of $L^1$, and has several useful and equivalent definitions.  We refer the reader to \cite{Grafakos} for a description of these spaces.  Since we will not be using the particulars of the definitions and only duality in our proofs, we will content ourselves with using the duality relationships as a definition. Note that with Hardy spaces we use sums instead of intersections; this can be seen as stemming from the role of the maximal functions in the definitions.

We will make use of the following duality relationships shortly.
\begin{prop}
    The dual of $H^1$ is $BMO$, and the dual of $VMO$ is $H^1$.  That is
    \begin{enumerate}
        \item $H^{1*} = BMO$
        \item $VMO^* = H^1$.
    \end{enumerate}
    These hold true both on $\RR$ and $\TT$. They also hold in the $n$-adic setting.
\end{prop}
The proof of $H^1$ duality is due to Fefferman-Stein \cite{FS} and VMO duality is due to Coifman-Weiss \cite{CW}.

We first show, via a simple construction, that the \emph{infinite} intersection of the $n$-adic $VMO$ spaces on the real line is never $VMO$.  The structure of $VMO$ permits such a claim, showing how different it is from the other function classes.
\begin{thm}\label{VMOallnumberstheorem1}
    $\bigcap_{n \in \NN} VMO_n \neq VMO$
\end{thm}
\begin{proof}
    Define $f$ such that:
    \begin{align*}
    f(x) = 
        \begin{cases}
        0, x < 0\\
        1, x \geq 0
    \end{cases}
    \end{align*}
    We observe $0$ is an endpoint of all $n$-adic intervals. So, when restricting to $n$-adic intervals, $f$ will be constant and thus $VMO_n$ for all $n \in \NN$. We show  $\lim_{\delta \rightarrow 0} \sup_{I: |I|< \delta}\fint_I |f-\fint_If| > 0$. Fix $\delta > 0$ to be arbitrarily small. Then let $J$ be the interval $(-r,r)$ such that $|J| < \delta$. Noting that $\fint_{J}f = \frac{1}{2r}\int_{J}f = \frac{1}{2}$, we compute:
    \begin{align*}
        \sup_{I: |I|< \delta}\fint_I |f-\fint_If| &\geq \fint_J |f - \fint_J f| 
        =\frac{\int_{(-r,0)} |0 - \frac{1}{2}| + \int_{(0,r)} |1 - \frac{1}{2}|}{2r} 
        = \frac{1}{2}.
    \end{align*}
   
   Since this estimate holds as $\delta\to 0$, then $f\notin VMO$.

\end{proof}

\begin{rmk}
    The exact same proof also shows that for any finite set $S \subseteq \NN$ that 
    \[
    \cap_{n\in S}VMO_n \neq VMO.
    \]
\end{rmk}
Now we can prove Theorem \ref{application theorem}.  This proof includes an alternative proof of the $VMO$ result in the above remark and actually much more, since it uses a much more expansive framework. Indeed this proof ties all the weight and function cases together with the measures constructed earlier via a functional analysis framework.

\begin{proof}[Proof of Theorem \ref{application theorem}]

We construct the function $w_\mu$ to be the weight associated with our measure $\mu$, defined at the top of this section.  Then  $w_\mu \in \bigcup_{n \in S}RH_r^n$, and $w_\mu \not\in RH_r$, by the remarks in the introduction and the fact that the construction of the measure is the same as in \cite{AH}. Note that the proof of this result for two distinct primes appears in the final section of \cite{AH}, including carefully computing bounds on the reverse H\"older constants.  The computations therein only rely on the construction of the measure, and the exhaustion procedure, and as these remain unchanged, the proof also follows through in this setting.  We leave the details to the interested reader. Thus, $RH_r^n \not\subseteq \bigcup_{n \in S}RH_r^n$.

We select our $r$ such that $\max_{n\in S}{1 - \frac{\ln n}{\ln 2}} < r < \infty$. By eplacing all instances of $r$ in the reverse H{\"o}lder weights case with $\frac{-1}{r-1}$, we get an identical argument, such that $A_r \not\subseteq \bigcap_{n\in S} A_r^n$.  See \cite{AH} for details.

We then apply the fact that for any $A_\infty$ weight $w$, $\log|w| \in BMO$ (the same holds in the $n$-adic case).  Let $f(x) = \log w_\mu(x)$. Since $w_\mu \in \bigcap_{n\in S} A_r^n$, we know $f \in \bigcap_{n \in S}BMO_n$. We then apply an identical argument to \cite{ATV} to observe that $w_\mu \not\in BMO$. Observe that our measure construction differs slightly from \cite{ATV}, but not in a way that substantially alters our proof. In particular, we get that $||f||_{BMO} \geq \frac{\alpha}{4} \log \frac{b}{a}$. Since $\alpha$ can be made to be arbitrarily large, we get that $||f||_{BMO}$ is unbounded.  Also in a similar manner as \cite{ATV}, we see that $f\notin BMO_n$ for any $n\in S$.

Now to show the Hardy space result we use duality.  A standard theorem from functional analysis says that if $X,X_1,\dots ,X_k,Y,Y_1,\dots ,Y_k$ are Banach spaces with $X^* = Y$ and $X_i^* = Y_i$, and if $Y \neq \cap_i Y_i$, then $X \neq \sum_i X_i$.  We apply this result first with $X = H^1$, $X_i = H^1_i$, $Y = BMO$ and $Y_i = BMO_i$ to conclude that 
\[
\sum_{i: n_i \in S}H^1_{n_i} \neq H^1.
\]
Note that the functional analysis result also holds under the following modification: if $Y \neq \sum_i Y_i$, then $X \neq \cap_i X_i$. We then apply this functional analysis result for a second time with $X = VMO$, $X_i = VMO_i$, $Y = H^1$, and $Y_i = H^1_i$ to conclude that 
\[
\sum_{i: n_i \in S}VMO_{n_i} \neq VMO.
\]
These results hold on $\RR$ and $\TT$ since duality is the same in both cases, and we gave a compact version of our measure, originally defined on $\RR$, in Section \ref{finite bases}.
\end{proof}

 \section{Appendix: Description of necessity of new construction and intermediate lemmas}
 \subsection{Number theory behind need for a new construction in Section \ref{finite bases}}
 Here we show, given points $k_i/p_i^{m_i}$ for primes $p_1,\ldots,p_n$ satisfying the criterion established in \cite{AH}, there is only a finite number of $q$-adic points ${j}/{q^n}$ such that \[\frac{k_i}{p_i^{m_i}}-\frac{j}{q^n}=\frac{1}{p_i^{m_i}q^n}\] for all $i.$ In \cite{AH}, it was shown that there are infinitely many $q$-adic points $\frac{j}{q^n}$ satisfying the inequality above for a single prime $p_i$, which guaranteed the existence of at least one $q$-adic interval sitting completely inside of the chosen $p$-adic interval. As the existence of such a $q$-adic interval was necessary for the construction of the measure in \cite{AH}, the number theory process developed in \cite{AH} and \cite{ATV} would be difficult to generalize to more than two bases.

The main criterion for interval selection in \cite{AH} is that inside any interval $\widetilde{J}$ and a chosen $\varepsilon>0$ we can choose a $q$-adic interval $I\subset \widetilde{J}$ such that $0<Y(J)-Z(I)<\varepsilon|I|$ where $J$ is the smallest $p$-adic interval containing $I.$ Now, in order to generalize the construction of this measure for finitely many primes (and then finitely many natural numbers altogether) we want to be able to select a $q$-adic interval $I$ inside any $\widetilde{J}$ such that $0<Y(J_i)-Z(I)<\varepsilon |I|$ with  $J_i$ denoting the largest $p_i$-adic interval containing $I$ for any collection of primes $p_i>q.$ In Proposition 3.5 in \cite{AH}, the method for constructing these intervals is to start with a $p$-adic interval $J$ defined by $Y(J)=\frac{k}{p^{m}}$ and $k\equiv q^{l(p-1)}\mod p^{m}.$ The authors assume $k\equiv q^{l(p-1)}\mod p^m$ for the convenience of applying their number theoretic results. From there, they construct an infinite sequence of $n$ and a corresponding $j$ for each $n$ such that \[\frac{k}{p^{m}}-\frac{j}{q^n}=\frac{1}{p^{m}q^n}\] so by choosing $\frac{1}{p^{m}}<\varepsilon q$ and $n$ they get that \[\left[\frac{j-q+1}{q^n},\frac{j+1}{q^n}\right)\subseteq \left[\frac{k-1}{p^{m}},\frac{k-1+p}{p^{m}}\right).\] 
The difficulty in extending this number theoretic procedure is then due to the need to find infinitely many solutions $n$ to a system of equations. That is, if given 
\[Y(J_1)=\frac{k_1}{p_1^{m_1}} \text{ and } Y(J_2)=\frac{k_2}{p_2^{m_2}}\] with $k_1\equiv q^{l_1(p_1-1)(p_2-1)}\mod p_1^{m_1}$ and $k_2\equiv q^{l_2(p_1-1)(p_2-1)}\mod p_2^{m_2}$ (the proper analogue to the $k\equiv q^{l(p-1)}$ criterion) then if 
\begin{equation} \label{syst} \frac{k_1}{p_1^{m_1}}-\frac{j}{q^n}=\frac{1}{p_1^{m_1}q^n} \end{equation}
and
\begin{equation} \label{syst2} \frac{k_2}{p_2^{m_2}}-\frac{j}{q^n}=\frac{1}{p_2^{m_2}q^n}\end{equation} 
 solving for $j$ gives
\[\frac{k_2q^n-1}{p_2^{m_2}}=\frac{k_1q^n-1}{p_1^{m_1}}\] so \[\left(\frac{k_2}{p_2^{m_2}}-\frac{k_1}{p_1^{m_1}}\right)q^n=\frac{1}{p_2^{m_2}}-\frac{1}{p_1^{m_1}}\] and thus \[n=\log_q\left(\frac{p_1^{m_1}-p_2^{m_2}}{k_2p_1^{m_1}-k_1p_2^{m_2}}\right).\] Since $q, p_1, p_2$ are fixed primes, and the remaining variables on the right-hand side of the equation above are fixed by the intervals $J_1$ and $J_2$, there is at most one integer solution $n$ satisfying the equations \ref{syst} and \ref{syst2}. From our discussion above, we see that the argument in \cite{AH} cannot be directly extended to more than two prime bases. 

\subsection{Topological lemmas}
\begin{prop}
    \label{torus 1}
    Let $\TT^\infty :=\TT^\omega$ denote the countable direct product of the torus $\RR/\ZZ.$ Then $\TT^\omega$  is metrizable and in particular the induced product topology is equivalent to the topology generated by the metric \[d(x,y)=\sum_{k=1}^\infty \frac{|x_k-y_k|}{2^k}.\]
\end{prop}

\begin{proof}
    This metric is likely well-known -- we found it in \cite{Sato}.  There the fact that the open balls from this metric are equivalent to the product topology on $\TT^\omega$ is presented without proof. For completeness, we give a proof.

    First, let $x\in \TT^\omega$ and $\varepsilon>0$ and take the open ball $B_\varepsilon(x)$ under the metric topology. We show that we can find some open set under the product topology which fits inside $B_\varepsilon(x).$, by taking the open intervals inside the first $n$ tori such that $\frac{n+1}{2^n}<\varepsilon$ with the open ball $|z_i-x_i|<\frac{2^i \varepsilon}{n+1}.$

    Next, take some open set under the product topology. These are going to be a countable union of finite products of open sets as these form a sub-basis of the product topology. Hence, it is sufficient to consider finite products of open sets. That is, we want to show that given some $U=U_{i_1}\times\cdots U_{i_k}$ where $U_{i_j}$ is an open set of the $i_j^{\text{th}}$ copy of $\TT$ we can find an open ball $B_\varepsilon(x)\subseteq U.$ Suppose without loss of generality that $U=B_{\varepsilon_1}(x_1)\times\cdots\times B_{\varepsilon_k}(x_k).$ Then, choose $y\in U$ and $\delta>0$ small enough such that if $|z_i-y_i|<\frac{2^i\delta}{k+1}$ then $|z_i-x_i|<\varepsilon_i$ for all $i.$ Hence, $B_\delta(y)\subseteq U.$
\end{proof}
\begin{prop}
\label{torus 2}
     Suppose $S\subseteq \TT^\omega$ has the property that $0$ is a limit point of $S$ inside any finite torus. Then, $0\in \overline{S}.$
\end{prop}
\begin{proof}
    Consider $B_\varepsilon(0),$ the open ball of radius $\varepsilon$ under $d.$ We show that there are infinitely many points of $S$ inside $B_\varepsilon(0).$

    Take $n$ such that $\frac{n+1}{2^{n}}<\varepsilon.$ Then, by the hypothesis there exists infinitely many points of $z\in S$ such that $|z_i|<2^i\frac{\varepsilon}{n+1}$ for $1\leq i \leq n$ as $S$ is dense in any finite torus.  Thus, for the chosen $z$ we get \[d(z,0)<\sum_{i=1}^n 2^i\frac{\varepsilon}{2^i(n+1)}+\sum_{i=n+1}^\infty\frac{|z_i|}{2^i} < n\frac{\varepsilon}{n+1}+\sum_{i=n+1}^\infty\frac{|z_i|}{2^i}\]\[<n\frac{\varepsilon}{n+1}+\frac{1}{2^n}<n\frac{\varepsilon}{n+1}+\frac{\varepsilon}{n+1}=\varepsilon.\]

    Hence we have infinitely many points inside this ball so therefore we have that $0$ is a limit point of $S$ inside the infinite torus.
\end{proof}
\begin{rmk}
    In effect all we are using in these two lemmas is the existence of a metric. We do not need to explicitly use the one given, only the triangle inequality.
\end{rmk}

\end{document}